\documentclass[review,hidelinks,onefignum,onetabnum]{siamart220329}
\usepackage[mathlines]{lineno}
\nolinenumbers
\usepackage{amsmath}
\usepackage{cases}
\usepackage{lipsum}
\usepackage{amsfonts}
\usepackage{graphicx}
\usepackage{epstopdf}
\usepackage{algorithmic}
\usepackage{mathrsfs}
\usepackage{subcaption}

\usepackage{dsfont,amssymb,amscd,latexsym,amsxtra,amsfonts}

\usepackage{hyperref}
\hypersetup{ colorlinks=true, linkcolor=blue, anchorcolor=blue,
    citecolor=blue}
\usepackage[round, sort]{natbib}

\usepackage{enumerate}
\usepackage{mathrsfs}
\usepackage{cases}

\usepackage{tikz}
\usetikzlibrary{calc,arrows}
\usepackage{verbatim}
\usepackage{epstopdf}
\usepackage{color}
\usepackage{bm}
\usepackage{caption}
\usepackage{subcaption}

\ifpdf
  \DeclareGraphicsExtensions{.eps,.pdf,.png,.jpg}
\else
  \DeclareGraphicsExtensions{.eps}
\fi

\renewcommand{\thefigure}{\arabic{figure}}

\newsiamremark{remark}{Remark}
\newsiamremark{hypothesis}{Hypothesis}
\crefname{hypothesis}{Hypothesis}{Hypotheses}
\newsiamthm{claim}{Claim}
\newsiamthm{assumption}{Assumption}

\headers{{ {Equilibrium portfolio selection under beliefs-dependent utilities}}}{Xiaochen Chen, Guohui Guan and Zongxia Liang}

\title{{ {Equilibrium portfolio selection under beliefs-dependent utilities}}\thanks{Submitted to the editors DATE.
\funding{This work was funded by the National Natural Science Foundation of China, grants 12271290 and 11871036.}}}

\author{Xiaochen Chen\thanks{Department of Mathematical Sciences, Tsinghua University, Beijing, 100084 People’s Republic of China\\
  (\email{xc-chen24@mails.tsinghua.edu.cn}).}
\and  Guohui Guan\thanks{Center for Applied Statistics, Renmin University of China, Beijing, 100872 People’s Republic of China
  (\email{guangh@ruc.edu.cn}).}
\and Zongxia Liang\thanks{Department of Mathematical Sciences, Tsinghua University, Beijing, 100084 People's Republic of China\\
  (\email{liangzongxia@tsinghua.edu.cn}).}
}

\usepackage{amsopn}

\ifpdf
\hypersetup{
  pdftitle={Equilibrium portfolio selection under beliefs-dependent utilities},
  pdfauthor={Xiaochen Chen, Guohui Guan and Zongxia Liang}
}
\fi
\usepackage[top=3.2cm, bottom=3.2cm, left=2.8cm, right=2.9cm]{geometry}

\begin{document}

\maketitle

\begin{abstract}
This paper investigates portfolio selection within a continuous-time financial market with regime-switching and beliefs-dependent utilities. The market coefficients and the investor's utility function both depend on the market regime, which is modeled by an observable finite-state continuous-time Markov chain. The optimization problem is formulated by aggregating expected certainty equivalents under different regimes, leading to time-inconsistency. Utilizing the equilibrium strategy, we derive the associated extended Hamilton-Jacobi-Bellman (HJB) equations and establish a rigorous verification theorem.  As a special case, we analyze equilibrium portfolio selection in a beliefs-dependent risk aversion model. In a bull regime, the excess asset returns, volatility, and risk aversion are all low, while the opposite holds in a bear regime. Closed-form solutions in the CRRA preference regime model of bull and bear markets are obtained, which is expressed by a solution to four-dimensional non-linear ODEs.  The global existence of the ODEs is proven and we verify the equilibrium solution rigorously. We show that the equilibrium investment strategy lies between two constant Merton's fractions. Additionally, in our numerical experiment, the equilibrium proportion allocated in the risky asset is greater in a bull regime than in a bear regime and the equilibrium proportion increases with time in a bull regime while decreasing in a bear regime.
\end{abstract}

\begin{keywords}
 Portfolio selection, Beliefs-dependent utilities, Regime-switching, Certainty equivalents, Equilibrium strategy, Time-inconsistency
\end{keywords}

\begin{MSCcodes}
 93E20, 49J15, 91B28
\end{MSCcodes}

\section{{{\bf Introduction}}}
Optimal portfolio selection has long been a fundamental issue in modern financial theory. Pioneering work by \citet{samuelson1969lifetime} and \cite{merton1969lifetime,merton1971optimum} establish foundational principles for subsequent research in this area. 
Over the decades, substantial advancements have emerged, as noted in \cite{kramkov1999asymptotic}, \cite{li2002dynamic}, \cite{brandt2010portfolio}, and \cite{wong2017utility}. A key characteristic of financial markets is that constant financial parameters typically apply only over short durations. This was particularly evident during the global financial crisis of 2008-2009 and the COVID-19 pandemic, where stock returns and volatilities exhibited distinct patterns across different periods. Research indicates that various market parameters-including the risk-free rate, expected returns on risky assets, and market volatility are influenced by prevailing market regimes (see, e.g., \cite{fama1989business}, \cite{schwert1989does}, \cite{french1987expected}, and \cite{hamilton1989new}). Consequently, many scholars have studied portfolio selection under regime-switching, which employs a continuous-time Markov process with a finite state space to capture the uncertainty of macroeconomic conditions. The expected utility maximization problem under regime-switching has been studied in works such as \cite{sass2004optimizing}, \cite{rieder2005portfolio}, \cite{brandt2010portfolio}, and \cite{liu2011dynamic}, etc.

Due to analytical tractability, the expected utility maximization problem under regime-switching assume a fixed utility function across different regimes (see \cite{brandt2010portfolio}, \cite{liu2014optimal}, \cite{gassiat2014investment}). However, in practice, an investor's utility and risk aversion are not static. 

Numerous studies indicate that investors typically display higher risk aversion during bull markets and lower risk aversion during bear markets. For instance, \cite{guiso2018time} provide empirical evidence of a significant surge in investor risk aversion following the 2008 financial crisis. Additionally, \cite{gordon2000preference} employ a two-state discrete-time Markov chain to model market regimes, positing that risk aversion varies with the regime. Their findings suggest substantial improvements in addressing the cyclical nature of equity prices. The optimization problem under beliefs-dependent (state-dependent) utilities can be formulated as follows:
\begin{equation}\label{problem:1}
\max_\pi \mathbb{E}\left[U^{\epsilon_T}\left(X^\pi_T\right)\,|\,\mathcal{F}_t\right],
\end{equation}
where $\pi$ represents the investment strategy, $X^\pi_T$ denotes the investor's wealth at terminal time $T$,  $\epsilon_t$ is a finite-state continuous-time Markov chain representing market regimes at time $t$, $U^{\epsilon_T}(\cdot)$ is the beliefs-dependent utility function and $\mathcal{F}_t$ represents the market information before time $t$. The studies by \cite{karni1983correspondence} and \cite{karni1983risk} extend the Arrow-Pratt measures of risk aversion in the context of beliefs-dependent utilities. Problem \eqref{problem:1}  is time-consistent, allowing for the application of dynamic programming methods and numerical methods. Beliefs-dependent utilities have been shown to be a powerful analytical tool, as evidenced by their application in \cite{chabi2008state} to address the risk aversion puzzle and in \cite{berrada2018asset} to capture the dynamic features of empirical financial data. Furthermore, \cite{li2022time} demonstrate that incorporating beliefs-dependent utilities significantly enhances the economic value of portfolio decisions.

Problem \eqref{problem:1} aggregates utilities across different regimes on a utility-scale. However, the differences between utilities in various regimes can be significant, and the optimal portfolio may be primarily influenced by performance in specific regimes. Utility functions are inherently relative, lacking nominal attributes, and are not directly comparable to other value functions. Nonetheless, they can be compared through certainty equivalents, which are expressed in monetary terms. Inspired by \cite{desmettre2023equilibrium}, we reformulate the optimal problem under beliefs-dependent utilities by aggregating the expected certainty equivalents on a monetary-scale, specifically:
\begin{equation}\label{problem:2}
\max_\pi
\mathbb{E}\left[(U^{\epsilon_T})^{-1}\left(\mathbb{E}\left[U^{\epsilon_T}\left(X^\pi_T\right)\,|\,\mathcal{F}_t,\epsilon_T\right]\right)\,|\,\epsilon_t\right].
\end{equation} 
By incorporating expected certainty equivalents, the objective in Problem \eqref{problem:2} loses its linearity compared to Problem \eqref{problem:1}, leading to time-inconsistency. As a result, maximizing Problem \eqref{problem:2} becomes exceedingly challenging. To tackle the issue of time-inconsistency, the intra-personal equilibrium approach under time-consistent planning is often employed. This paper therefore introduces the concept of Nash equilibrium to find an equilibrium solution to the optimization problem, rather than pursuing a classical optimal solution that merely maximizes the objective function. The equilibria authorize the investor to take into account the behavior of future selves and to seek the control strategy by game-theoretic thinking, thus ensuring the time-consistency of the chosen strategy.

The concept of equilibrium strategy originates from \cite{strotz1955myopia}, who demonstrated that sophisticated investors inherently adopt an equilibrium investment strategy. The formal mathematical definition of equilibrium investment strategies in continuous time was later established by \cite{ekeland2006being}, \cite{ekeland2008investment}, and \cite{ekeland2012time}, with non-exponential discounting serving as a key example. Furthermore, equilibrium strategies have been applied to the mean-variance problem, as explored by \cite{bjork2014theory}, \cite{bjork2014mean}, and \cite{bjork2017time}. Both the stochastic risk aversion model in \cite{desmettre2023equilibrium} and the smooth ambiguity preference model in \cite{guan2024equilibrium} face the challenge of time-inconsistency due to the non-linearity of their objective functions, and both address this issue through an equilibrium strategy approach. 

In this paper, we consider a financial market consisting of a risk-free asset and a risky asset. Both the market coefficients and the utilities depend on the market regime $\epsilon$. In the case of beliefs-dependent utilities, Problem \eqref{problem:2} exhibits time-inconsistency. By analyzing Problem \eqref{problem:2} under regime-switching, this paper makes the following main contributions:

(1) To our knowledge, most literature on beliefs-dependent utilities aggregates performances across different regimes on a utility scale, focusing on Problem \eqref{problem:1}. However, because utility functions are not directly comparable, this paper aggregates performances across  different regimes on a monetary scale. A similar approach is taken in \cite{desmettre2023equilibrium}, which addresses an equilibrium portfolio selection problem with random risk aversion but does not incorporate regime-switching. When \(\epsilon_T\) is treated as a static random variable and market coefficients are constant, our work reduces to the formulation presented in \cite{desmettre2023equilibrium}. In contrast, our paper introduces regime-switching and links random risk aversion to the regime, leading to greater mathematical complexity while providing a betetr economic interpretation. (2)  Unlike Problem \eqref{problem:1}, Problem \eqref{problem:2} is time-inconsistent. By introducing equilibrium strategies, this paper addresses the time-inconsistency arising from the certainty equivalents in the objective function. We derive the extended HJB equations and provide a rigorous verification theorem. (3) Additionally, we analyze a concrete example within a preference regime model of bull and bear markets, where \(U^{\epsilon_T}(\cdot)\) follows a CRRA utility function and \(\epsilon_t \in \{1, 2\}\) takes on two values. Using the separation of variables, we obtain a closed-form solution expressed by four-dimensional non-linear ODEs. We demonstrate the global existence of these ODEs and rigorously verify the equilibrium solution. (4) Due to the lack of homogeneity in beliefs-dependent utilities in Problem \eqref{problem:1}, obtaining closed-form solutions is challenging, leading most research to rely on numerical methods (see, e.g., \cite{chabi2008state}, \cite{berrada2018asset}, \cite{li2022time}). To our knowledge, \cite{sotomayor2009explicit} is the only paper that provides explicit solutions for beliefs-dependent utilities in a specific case. However, the beliefs-dependent utilities in \cite{sotomayor2009explicit} are restricted to a constant multiplied by a root function, with the constant varying across regimes. Consequently, the risk aversion coefficients remain the same across different regimes, failing to capture the differing risk aversion attitudes of investors during bear and bull markets. Thus, the results from \cite{sotomayor2009explicit} indicate that utility differences do not influence optimal investment strategies across regimes. Modifying Problem \eqref{problem:1} to Problem \eqref{problem:2}, we can maintain homogeneity and derive explicit solutions, enabling us to analyze the effects of varying relative risk aversion coefficients on equilibrium strategies across different regimes.  (5) When $U^{\epsilon_T}(\cdot)=U(\cdot)$ does not depend on $\epsilon_T$ and is belief independent, we see that although the objective functions in Problems \eqref{problem:1} and \eqref{problem:2} are different, the solutions are the same and an explanation is provided. In the expected utility maximization problem under regime-switching, the optimal strategy is typically represented by Merton's fraction multiplied by a constant. In contrast, our work expresses the equilibrium investment strategy as Merton's fraction multiplied by a deterministic function that increases over time in a bull regime and decreases over time in a bear regime in the numerical experiment. This equilibrium investment strategy falls between two Merton's fractions. Additionally, our numerical results show that the equilibrium proportion allocated to the risky asset is greater in a bull regime than in a bear regime.

The remainder of the paper is organized as follows. Section \ref{sec2} presents the financial model and the optimization problem. Section \ref{sec3} establishes a verification theorem. Section \ref{sec4} provides closed-form results in the CRRA preference model of bull and bear markets. Section \ref{sec5} presents the numerical examples and economic analysis.

\section{\bf Financial model and problem formulation}\label{sec2}
Let $(\Omega,\mathcal{F},\{\mathcal{F}_t\}_{t\in[0,T]},\mathbb{P})$ be a filtered complete probability space,  where constant $T>0$ is the time horizon and the filtration $\{\mathcal{F}_t\}_{t\in[0,T]}$ satisfies the usual conditions, representing the whole information of the financial market. All the processes introduced below are supposed to be well defined and adapted to the filtration $\{\mathcal{F}_t\}_{t\in[0,T]}$. 
\subsection{\bf The financial model}
This paper considers a financial market under regime-switching. Let $ \epsilon= \{\epsilon_t\}_{t\in[0,T]}$ be a continuous-time, stationary, finite-state Markov chain adapted to the filtration $\{\mathcal{F}_t\}_{t\in[0,T]}$  with the state space $S=\{1 ,2,\cdots , N\}$. Here $\epsilon_t$ represents the regime of the financial market at time $t$, and $N$ is the number of regimes. Furthermore, let $Q=[q_{ij}]_{N\times N}$ be the generator of the continuous-time Markov chain$\{\epsilon_t\}_{t\in[0,T]}$, where
$$q_{ii}=-\lambda_i\le0;\;\ \  q_{ij}\ge0,\;\ \  \forall\,i\ne j;\;\ \  \sum_{j\in S}q_{ij}=0.$$

The financial market consists of two assets: one bond (risk-free asset) and one stock (risky asset). Let the $\{\mathcal{F}_t\}_{t\in[0,T]}$-adapted processes $P_0=\{P_0(t)\}_{t\in[0,T]}$ and $P_1=\{P_1(t)\}_{t\in[0,T]}$ represent the price of the bond and the stock, respectively. They satisfy the stochastic differential equations (abbr. SDEs): 
	\begin{equation}
        \left\{
	    \begin{aligned}
	    & d P_0(t)=r_{\epsilon_t}P_0(t)dt, & P_0(0)=1,\\
		& d P_1(t)=P_1(t)(\mu_{\epsilon_t}d t+\sigma_{\epsilon_t} d W_t),  &P_1(0)=p_1>0,
	    \end{aligned}
	    \right.
	\end{equation}
where $W=\{W_t\}_{t\in[0,T]}$ is a standard Brownian motion with respect to filtration  $\{\mathcal{F}_t\}_{t\in[0,T]}$ and $r_i,r_i;\sigma_i,\sigma_i;\mu_i,\mu_i>0,i\in S$, represent the expected rate of returns on the risk-free assets, the volatility of the risky assets and the rates of return for the risky asset under the market regime $i$, respectively. We assume that the stochastic processes $W$ and $\epsilon$ are independent.

\begin{remark}
    Considering that $\{\epsilon_t\}_{t\in[0,T]}$ represents the overall market state, directly influencing the expected returns and volatility of assets in the market, and $\{W_t\}_{t\in[0,T]}$ specifically affects the volatility of the risky asset given the regime, it is reasonable to assume that  $W$ and $\epsilon$  are independent.
\end{remark}

\subsection{\bf The wealth process and trading strategy}
Let the function $$\pi:(t,x,i)\to \pi(t,x,i)$$ denote the trading strategy, which represents  the dollar amount invested  in the risky asset at time $t$, when the total wealth is $x$ and the market regime is $i$. The wealth process under the trading  strategy $\pi$ is denoted by $X^{\pi}$, satisfying the following SDE:
\begin{equation}\label{wealth}
    d X_t^\pi=r_{\epsilon_t} X_t^\pi d t+(\mu_{\epsilon_t}-r_{\epsilon_t})\pi(t,X_t^\pi,\epsilon_t) d t+\pi(t,X_t^\pi,\epsilon_t)\sigma_{\epsilon_t} d W_t,\qquad 0\le t\le T.
\end{equation}
Let 
\begin{equation}
    \Pi_0=\{\pi\, |\,  \pi:  [0,T]\times \mathbb{R}\times S\to \mathbb{R} \  \mbox{such that SDE \eqref{wealth} has a unique  solution} \;X^\pi\}.
\end{equation}
$\Pi\subseteq \Pi_0$ represents the set of admissible trading strategies $\pi$. An admissible strategy should also satisfy additional conditions, which are related to the specific form of the optimization problem. These conditions are temporarily omitted here and will be provided later in the context of the specific problem.

\subsection{\bf Optimal problem and equilibrium strategy}
Let $u^j:[0,\infty)\to\mathbb{R},j\in S$, denote the utility function under the market regime $j$, where $u^j$ is a strictly concave, strictly increasing and differential function, $\forall \; j\in S $.
This paper considers the following objective function:
\begin{equation}\label{objective function}
    J^\pi(t,x,i)\mathop{=}\limits^{\Delta}\sum_{j\in S}p(t,i,j)(u^j)^{-1}(\mathbb{E}_{t,x,i,j}[u^j(X^\pi_T)]),
\end{equation}
where $\mathbb{E}_{t,x,i,j}(\cdot)=\mathbb{E}[\cdot\,|\,X_t=x,\,\epsilon_t=i,\,\epsilon_T=j]$ and $p(t,i,j)=\mathbb{P}(\epsilon_T=j\,|\,\epsilon_t=i)$. The objective function \eqref{objective function} represents the weighted average of the certainty equivalents of the
terminal wealth across different market regimes. We consider maximizing the objective function, i.e., to consider the following optimization problem:
\begin{equation}\label{IP}
    \mathop{sup}\limits_{\pi \in \Pi}\left\{J^\pi(t,x,i)\right\}.
\end{equation}
However, this problem is time-inconsistent due to the  certainty equivalents, making it extremely difficult to directly maximize $J^\pi (t,x,i)$. Therefore, this paper considers the equilibrium solution to this problem rather than maximizing it  based on  the concepts of equilibrium strategy proposed by  \cite{strotz1955myopia},  and  the mathematical approaches of equilibrium established by   \cite{ekeland2006being}, \cite{ekeland2008investment}, and \cite{ekeland2012time}.
\begin{definition}
(Equilibrium strategy). Consider an admissible strategy $\hat{\pi}\in\Pi$. We say $\hat{\pi}$ is an equilibrium strategy, if for every $(t,x,i)\in[0,T]\times \mathbb{R}\times S$ and $\pi\in\Pi$, whenever $\pi_h\in\Pi$ for all sufficiently small $h>0$, we have
\begin{equation}\label{op1}
    \mathop{limsup}\limits_{h\to 0}\frac{J^{\pi_h}(t,x,i)-J^{\hat\pi}(t,x,i)}{h}\le 0,
\end{equation}
where$$ \pi_h(s)\mathop{=}\limits^{\Delta}\left\{
	\begin{aligned}
		\pi(s) & \quad s\in[t,t+h), \\
		\hat{\pi}(s) & \quad s\in[t+h,T]. 
	\end{aligned}
	\right.$$
\end{definition}
This paper is to determine the equilibrium strategy for the investment problem (\ref{IP}).
\begin{remark}
     \cite{rieder2005portfolio} investigated portfolio optimization problems under regime-\\ switching model. However, their utility function is independent of the terminal regime, resulting in the time-consistency. Therefore it is feasible to utilize dynamic programming principles to derive the Hamilton-Jacobi-Bellman equation and find the optimal trading strategy that maximizes expected utility. In contrast, the objective function in our study is non-linear due to  the  certainty equivalents, which yields the time-inconsistency. In addition, our utility function also  depends on the terminal regime. From a practical perspective, when the market is bullish, investors tend to exhibit lower risk aversion compared to periods of market downturn. Therefore,  the model in this paper is more realistic, and completely different from that  of \cite{rieder2005portfolio} .
\end{remark}
\begin{remark}
     \cite{desmettre2023equilibrium} explored the problem of random risk aversion. Although they did not consider a model with regime-switching, their study synthesized the impact of different utility functions at the terminal time by employing the certainty equivalent approach. While leading to time-inconsistency for maximizing the objective function, this approach provided stronger economic insights. The certainty equivalent also ensured the ``homogeneity'' of the objective function, providing mathematical convenience in solving the corresponding differential equations. Our paper considers beliefs-dependent utility functions, i.e., the utility functions depend on market regimes at time $T$. 
\end{remark}
\vskip 10pt
\section{\bf Verification Theorem}\label{sec3}
In this section, we derive the extended HJB equations of Problem \eqref{IP} and present a rigorous verification theorem.
\subsection{\bf Verification theorem}For every $\pi\in\Pi$, the infinitesimal generator $\mathcal{A}^{\pi}$ for the process $X^{\pi}$ defined by  SDE \eqref{wealth} is 
$$\begin{aligned}
\mathcal{A}^{\pi}f^{i,j}(t,x)&\mathop{=}\limits^{\Delta}f_t^{i,j}(t,x)+\frac{1}{2}f^{i,j}_{xx}(t,x)\sigma_i^2\pi^2(t,x,i)+f^{i,j}_x(t,x) (\mu_i-r_i)\pi(t,x,i)
\\&+f^{i,j}_x(t,x) xr_i-\lambda_i f^{i,j}(t,x)+\sum_{k\in S\backslash \{i\}}q_{ik}f^{k,j}(t,x)
\end{aligned}$$
 for $\{f^{i,j}(t,x)\}_{i,j\in S}\subset C^{1,2}([0,T)\times \mathbb{R}^+)\cap C([0,T]\times \mathbb{R}^+)$.

We make the following assumptions on functions $\{f^{i,j}(t,x)\}_{i,j\in S}$ , which will be used later in the verification theorem.
\begin{assumption}\label{assume1}
$\forall \pi \in \Pi,\;(t,x,i,j)\in[0,T)\times \mathbb{R}^+ \times S^2,\;\exists \,\hat{t}\in(t,T)$ such that,  under the conditional probability $\mathbb{P}_{t,x,i,j}$, the process $$\left\{\int^s_t f_x^{\epsilon_u,j}(u,X_u^\pi)\pi(u,X_u^\pi,\epsilon_u)\sigma_{\epsilon_u}d W_u\right\}_{s\in[t,\hat{t}]}$$
is a supermartingale  with respect to $\{\mathcal{F}_t\}_{t\in[0,T]}$.
\end{assumption}
\begin{assumption}\label{assume2}
$\forall \pi \in \Pi, (t,x,i,j)\in[0,T)\times \mathbb{R}^+ \times S^2,\;\exists \,\hat{t}\in(t,T)$ such that, under the conditional probability $\mathbb{P}_{t,x,i,j}$ , the process
$$\left\{\frac{1}{h} \int^{t+h}_t \mathcal{A}^\pi f^{\epsilon_u,j}(u,X_u^\pi)d u\right\}_{h\in(0,\hat{t}-t]}$$
is uniformly integrable.
\end{assumption}
\begin{assumption}\label{assume3}
$\forall \pi \in \Pi,\;(t,x,i,j)\in[0,T)\times \mathbb{R}^+ \times S^2$, we have
	$$\mathop{lim}\limits_{h\to 0^+} \mathbb{E}_{t,x,i,j}[f^{\epsilon_{t+h},j}(t+h,X^\pi(t+h))]=f^{i,j}(t,x).$$
\end{assumption}
    It is necessary to note that, compared to standard stochastic control problems, the conditional probability $\mathbb{P}_{t,x,i,j}$ includes the information about $\epsilon_T = j$ at the terminal time. For the subsequent derivations, it is needed to clarify that under the conditional probability $\mathbb{P}_{t,x,i,j}$, $ W=\{W_t\}_{t\in[0,T]}$ remains a standard Brownian motion with respect to $\{\mathcal{F}_t\}_{t\in[0,T]}$. This fact can be proven by using the independent increment property of $\{W_t\}_{t\in[0,T]}$ and the fact that $W$ and $\epsilon$ are independent.
We now state the verification theorem as follows:
\begin{theorem}\label{verification}(Verification Theorem)
Assume that  a family of functions $$\{f^{i,j}(t,x)\}_{i,j\in S}\subset C^{1,2}([0,T)\times \mathbb{R}^+)\cap C([0,T]\times \mathbb{R}^+)$$ and an admissible strategy $\hat\pi\in\Pi$ satisfy the following conditions:
	\\$(a)\;\forall (t,x,i,j)\in[0,T)\times \mathbb{R}^+ \times S^2$, the extended HJB equations hold
 \begin{eqnarray}\label{(a)}
 \left\{
	\begin{aligned}
	   	    &\mathop{sup}\limits_{\pi\in\Pi} \left\{ \sum_{k\in S}((u^k)^{-1})^{'}(f^{i,k}(t,x))(\mathcal{A}^\pi f^{i,k}(t,x))p(t,i,k)\right\}=0,  \\
	&\mathcal{A}^{\hat\pi}f^{i,j}(t,x)=0,\\
    &f^{i,j}(T,x)=u^j(x); 
	\end{aligned} \right.
	\end{eqnarray}
$(b)\;\forall (t,x,i,j)\in[0,T)\times \mathbb{R}^+  \times S^2$, the process $$\left\{\int^s_t f_x^{\epsilon_u,j}((t,X^{\hat\pi}_u)\hat\pi(t,X^{\hat\pi}_u,\epsilon_u)\sigma_{\epsilon_u}d W_u\right\}_{s\in[t,T]}$$ is a martingale with respect to $\{\mathcal{F}_t\}_{t\in[0,T]}$;
\\$(c)\; \mbox{  the family of the functions } \{f^{i,j} (t,x)\}_{i,j\in S}$ satisfies Assumptions  \ref{assume1} - \ref{assume3}.
\\Then $\hat\pi$ is an equilibrium strategy,  the function $ f^{i,j} (t,x)$ is 
\begin{equation}
	    f^{i,j} (t,x)=\mathbb{E}_{t,x,i,j}[u^j(X^{\hat\pi}_T)],\;\forall\,(t,x,i,j)\in[0,T)\times \mathbb{R}^+ \times S^2, 
	\end{equation}
and  the objective function under the equilibrium strategy $\hat\pi$ is
\begin{equation}
	    J^{\hat\pi} (t,x,i)=\sum_{j\in S }p(t,i,j)(u^j)^{-1}(f^{i,j} (t,x)),\;\forall\,(t,x,i)\in[0,T)\times \mathbb{R}^+ \times S.
	\end{equation}
\end{theorem}
\begin{proof}
    First, fix $(t,x,i)\in [0,T)\times\mathbb{R}^+ \times S$. Using SDE \eqref{wealth} and applying  the Itô's formula for Markov-modulated processes (see, e.g., \cite{sotomayor2009explicit}) to the process  $\{f^{\epsilon_t,j} (t,X_t^{\hat\pi})\}$, we have \begin{equation}\label{ito1}
      d f^{\epsilon_t,j}(t,X_t^{\hat\pi})=\mathcal{A}^{\hat\pi}f^{\epsilon_t,j}(t,X_t^{\hat\pi})d t+f_x^{\epsilon_t,j}(t,X_t^{\hat\pi})\hat\pi(t,X^{\hat\pi}_t,\epsilon_t)\sigma_{\epsilon_t}d W_t+d M_t\;,\forall\, j\in S, 
      \end{equation}
      where the process $M=\{M_t: t\leq T \}$ is a square-integrable martingale. Integrating \eqref{ito1} over the interval $[0,T]$ and then taking the conditional expectation $\mathbb{E}_{t,x,i,j}[\cdot]$, using Conditions $(a)$ and $(b)$, we obtain 
      \begin{equation}\label{obj}
      f^{i,j} (t,x)=\mathbb{E}_{t,x,i,j}[u^j(X^{\hat\pi}_T)].
  \end{equation}
  As such,
  \begin{equation}\label{objective1}
      J^{\hat\pi} (t,x,i)=\sum_{j\in S}p(t,i,j)(u^j)^{-1}(f^{i,j} (t,x)).
  \end{equation}
  Similarly, considering $\forall\,\pi\in\Pi$ and the corresponding $\pi_h$, integrating \eqref{ito1} over the interval $[t+h,T]$, we also have  \begin{equation*}
      \mathbb{E}_{t,x,i,j}[u^j(X^{\pi_h}_T)]=\mathbb{E}_{t,x,i,j}[f^{\epsilon_{t+h},j}(t+h,X^\pi(t+h))],
  \end{equation*}
  then 
  \begin{equation}
  \begin{aligned}\label{objective2}
  J^{\pi_h}(t,x,i)&=\sum_{j\in S} p(t,i,j)(u^j)^{-1}(\mathbb{E}_{t,x,i,j}[u^j(X^{\pi_h}_T)])\\&=\sum_{j\in S} p(t,i,j)(u^j)^{-1}(\mathbb{E}_{t,x,i,j}[f^{\epsilon_{t+h},j}(t+h,X^\pi(t+h))]).
  \end{aligned}
  \end{equation}
Using the Itô's formula for Markov-modulated processes over $f^{\epsilon_t,j} (t,X_t^{\pi})$ again, 
\begin{equation}\label{ito2}
      d f^{\epsilon_t,j}(t,X_t^{\pi})=\mathcal{A}^{\pi}f^{\epsilon_t,j}(t,X_t^{\pi})d t+f_x^{\epsilon_t,j}(t,X_t^{\pi})\pi(t,X^{\pi}_t,\epsilon_t)\sigma_{\epsilon_t}d W_t+d N_t,\;\forall\; j\in S,
  \end{equation}
  where $N=\{N_t: t\leq T\}$ is a square-integrable martingale. 
  Integrating \eqref{ito2} over the interval $[t,t+h]$ and then taking the conditional expectation $\mathbb{E}_{t,x,i,j}[\cdot]$, combining  Assumption \ref{assume1}, we get \begin{equation}
      \mathbb{E}_{t,x,i,j}[f^{\epsilon_{t+h},j}(t+h,X^\pi(t+h))]-f^{i,j}(t,x)\le \mathbb{E}_{t,x,i,j}\Bigg[\int^{t+h}_t \mathcal{A}^\pi f^{\epsilon_u,j}(u,X_u^\pi)d u\Bigg].
  \end{equation}
  Moreover, using  Assumption \ref{assume2}, 
  \begin{equation*}
      \mathbb{E}_{t,x,i,j}[f^{\epsilon_{t+h},j}(t+h,X^\pi(t+h))]-f^{i,j}(t,x)\le h\mathcal{A}^\pi f^{i,j}(t,x)+o(h).
  \end{equation*}
  Using that $(u^j)^{-1}$ is convex, we have 
  \begin{equation}\label{calculus1}
      \begin{aligned}
      &\frac{(u^j)^{-1}(\mathbb{E}_{t,x,i,j}[f^{\epsilon_{t+h},j}(t+h,X^\pi(t+h))])-(u^j)^{-1}(f^{i,j}(t,x))}{h}
      \\
       \le& \frac{1}{h}((u^j)^{-1})^{'}(\mathbb{E}_{t,x,i,j}[f^{\epsilon_{t+h},j}(t\!+\!h,X^\pi(t\!+\!h))])(\mathbb{E}_{t,x,i,j}[f^{\epsilon_{t\!+\!h},j}(t\!+\!h,X^\pi(t\!+\!h))]-f^{i,j}(t,x))
       \\
       \le& ((u^j)^{-1})^{'}(\mathbb{E}_{t,x,i,j}[f^{\epsilon_{t+h},j}(t+h,X^\pi(t+h))])(\mathcal{A}^\pi f^{i,j}(t,x)+o(1)).
      \end{aligned}
  \end{equation}
  Therefore, based on Eq.~\eqref{objective1}, Ineqs \eqref{objective2} and \eqref{calculus1} as well as Assumption \ref{assume3}, we have
  \begin{equation}
      \mathop{limsup}\limits_{h\to 0}\frac{J^{\pi_h}(t,x,i)-J^{\hat\pi}(t,x,i)}{h}\le\sum_{j\in S}((u^j)^{-1})^{'}(f^{i,j}(t,x))(\mathcal{A}^{\pi}f^{i,j}(t,x))p(t,i,j).
  \end{equation}
  Finally, using Condition ($\ref{(a)}$), 
\begin{equation}
      \mathop{limsup}\limits_{h\to 0}\frac{J^{\pi_h}(t,x,i)-J^{\hat\pi}(t,x,i)}{h}\le 0,
  \end{equation}
  thus $\hat\pi$ is an equilibrium strategy, and other results easily follow from (\ref{obj}) and (\ref{objective1}).
\end{proof}
\subsection{\bf The equilibrium feedback function} Considering Condition (a) of Theorem \ref{verification}, we get  
\begin{equation}\label{strategy1}
    \hat\pi\in\mathop{argsup}\limits_{\pi\in\Pi}\left\{\sum_{j\in S}((u^j)^{-1})^{'}(f^{i,j}(t,x))(\mathcal{A}^\pi f^{i,j}(t,x))p(t,i,j)\right\}.
\end{equation}
The right-hand side of \eqref{strategy1} is a quadratic function of  $\pi$. We make the following assumption:
\begin{assumption}\label{assume4}
$\forall (t,x,i,j)\in[0,T)\times \mathbb{R}^+ \times S^2,\ f^{i,j}_{xx}(t,x)<0$.
\end{assumption}
If Assumption \ref{assume4} holds, then, based on (\ref{strategy1}), the coefficient of the quadratic term is 
\begin{equation}
    \frac12\sum_{j\in S}(u^j)^{-1})^{'}(f^{i,j}(t,x))f^{i,j}_{xx}(t,x)\sigma_i^2p(t,i,j)<0.
\end{equation} 
Maximizing the right-hand side of \eqref{strategy1} and then using the first-order condition, we have that a candidate equilibrium strategy $\hat\pi$ satisfies \begin{equation}\label{strategy3}
    \hat\pi(t,x,i)=-\frac{\mu_i-r_i}{\sigma_i^2}\frac{\sum_{j\in S} p(t,i,j)((u^j)^{-1})^{'} (f^{i,j}(t,x))f_x^{i,j}(t,x)}{\sum _{j\in S} p(t,i,j)((u^j)^{-1})^{'} (f^{i,j}(t,x))f_{xx}^{i,j}(t,x)}.
\end{equation}
In what follows, we will use the expression  \eqref{strategy3} and verification theorem \ref{verification} to show in detail that the candidate  $\hat\pi$  is the equilibrium strategy of the original problem (\ref{IP})-(\ref{op1}) in the context of specific examples. 
\vskip 15pt
\section{\bf A Special Case: Two regimes and CRRA utility}\label{sec4}
In this section, we solve the problem (\ref{IP})-(\ref{op1})  in a CRRA preference model of bull and bear markets.
\subsection{\bf The special model} 

Now we consider the case of two regimes, i.e., $N=2,S=\{1,2\}$, where $\epsilon_t=1,2$ represent bull market and bear market, respectively. Additionally, we assume that the utility function $u^j$, $j\in S$,  is constant-relative-risk aversion (CRRA case), i.e.,
$$u^1(x)=\frac{1}{1-\alpha_1}x^{1-\alpha_1},\;u^2(x)=\frac{1}{1-\alpha_2}x^{1-\alpha_2},$$ 
where $\alpha_1,\alpha_2\in\mathbb{R}^+\backslash\{1\} $ are coefficients of relative-risk-aversion. 
\begin{definition}\label{admissible}
Given the above conditions: $N,\  S, \   u^i,\  i=1,2$, we say that  $\pi : [0,T]\times \mathbb{R} \times S\to\mathbb{R}$ is an admissible strategy, if it satisfies the following two conditions:\\
(i) For every $(t,x,i)\in [0,T)\times \mathbb{R}^+  \times S$, SDE \eqref{wealth} has a unique solution $X^\pi$, and $X_s^\pi>0\;a.s.,\,\forall s\in[t,T]$;\\
(ii) There exists a constant $C>0$ such that for every $(t,x,i)\in [0,T]\times \mathbb{R} \times S
$, $|\pi(t,x,i)|\le C|x|.$
\end{definition}
Let $\Pi$ denote the set of all admissible strategies.
In order to obtain the equilibrium strategy, based on the verification theorem \ref{verification} and \eqref{strategy1}, we need to solve the following boundary value problem for a system of partial differential equations (PDEs): 
\begin{eqnarray}\label{pdes}
\left\{
\begin{aligned}
    &\mathcal{A}^{\hat\pi}f^{i,j}=0,\;\forall i,j\in\{1,2\},\\
    &f^{i,j}(T,x)=\frac{1}{1-\alpha_j}x^{1-\alpha_j},\;\forall i,j\in\{1,2\},
\end{aligned}\right.
\end{eqnarray} 
where
\begin{equation*}
    \hat\pi(t,x,i)=-\frac{\mu_i-r_i}{\sigma_i^2}\frac{p(t,i,1)((u^1)^{-1})^{'} (f^{i,1})f_x^{i,1}+p(t,i,2)((u^2)^{-1})^{'} (f^{i,2})f_x^{i,2}}{p(t,i,1)((u^1)^{-1})^{'} (f^{i,1})f_{xx}^{i,1}+p(t,i,2)((u^2)^{-1})^{'} (f^{i,2})f_{xx}^{i,2}}.
\end{equation*}
\subsection{\bf Ansatz} We make the ansatz for  the solution of the PDEs \eqref{pdes}, i.e.,  Functions $\{f^{i,j}(t,x)\}_{i,j\in S}$ has the following  form:
\begin{equation}\label{solution1}
    f^{i,j}(t,x)=\frac{1}{1-\alpha_j}x^{1-\alpha_j}(g^{i,j}(t))^{\alpha_j},\;\forall\,i,j\in S.
\end{equation}
Then \begin{equation}\label{strategy2}
    \hat\pi(t,x,i)=\frac{\mu_i-r_i}{\sigma_i^2}\, x \, \frac{p(t,i,1)(g^{i,1}(t))^{\frac{\alpha_1}{1-\alpha_1}}+p(t,i,2)(g^{i,2}(t))^{\frac{\alpha_2}{1-\alpha_2}}}{\alpha_1 p(t,i,1)(g^{i,1}(t))^{\frac{\alpha_1}{1-\alpha_1}}+\alpha_2 p(t,i,2)(g^{i,2}(t))^{\frac{\alpha_2}{1-\alpha_2}}}.
\end{equation}
Define
\begin{equation}
A_i(t)\mathop{=}\limits^{\Delta}\frac{p(t,i,1)(g^{i,1}(t))^{\frac{\alpha_1}{1-\alpha_1}}+p(t,i,2)(g^{i,2}(t))^{\frac{\alpha_2}{1-\alpha_2}}}{\alpha_1 p(t,i,1)(g^{i,1}(t))^{\frac{\alpha_1}{1-\alpha_1}}+\alpha_2 p(t,i,2)(g^{i,2}(t))^{\frac{\alpha_2}{1-\alpha_2}}},\;\forall\; i\in S.
\end{equation}
Then the boundary value problem for the PDEs \eqref{pdes}  is transformed into the following initial value problem for a system of non-linear ODEs:
\begin{equation}\label{odes1}
   \!\!\!\!  \!\!   \!\!  \left\{
	\begin{aligned}
		&\frac{\alpha_1}{1\!-\!\alpha_1}g_t^{1,1}(t)\!=\!\Bigg\{\frac{1}{2}\alpha_1 (\frac{\mu_1\!-\!r_1}{\sigma_1})^2 A_1^2(t)\!-\!(\frac{\mu_1\!-\!r_1}{\sigma_1})^2 A_1(t)
		\!-\!r_1\!+\!\frac{\lambda_1}{1\!-\!\alpha_1}\!\Big[1\!-\!(\frac{g^{2,1}(t)}{g^{1,1}(t)})^{\alpha_1}\Big]\Bigg\}g^{1,1}(t),
		\\
		&\frac{\alpha_1}{1\!-\!\alpha_1}g_t^{2,1}(t)\!=\!\Bigg\{\frac{1}{2}\alpha_1 (\frac{\mu_2\!-\!r_2}{\sigma_2})^2A_2^2(t)\!-\!(\frac{\mu_2\!-\!r_2}{\sigma_2})^2 A_2(t)
		\!-\!r_2\!+\!\frac{\lambda_2}{1\!-\!\alpha_1}\Big[1\!-\!(\!\frac{g^{1,1}(t)}{g^{2,1}(t)})^{\alpha_1}\Big]\Bigg\}g^{2,1}(t),  
		\\
		&\frac{\alpha_2}{1\!-\!\alpha_2}g_t^{1,2}(t)\!=\!\Bigg\{\frac{1}{2}\alpha_2 (\frac{\mu_1\!-\!r_1}{\sigma_1})^2\! A_1^2(t)\!-\!(\frac{\mu_1\!-\!r_1}{\sigma_1})^2 A_1(t)
		\!-\!r_1\!+\!\frac{\lambda_1}{1\!-\!\alpha_2}\Big[1\!-\!(\frac{g^{2,2}(t)}{g^{1,2}(t)})^{\alpha_2}\Big]\Bigg\}g^{1,2}(t), 
		\\
		&\frac{\alpha_2}{1\!-\!\alpha_2}g_t^{2,2}(t)\!=\!\Bigg\{\frac{1}{2}\alpha_2 \!(\frac{\mu_2\!-\!r_2}{\sigma_2})^2A_2^2(t)\!-\!(\frac{\mu_2\!-\!r_2}{\sigma_2})^2 A_2(t)
		\!-\!r_2\!+\!\frac{\lambda_2}{1\!-\!\alpha_2}\Big[1\!-\!(\!\frac{g^{1,2}(t)}{g^{2,2}(t)})^{\alpha_2}\Big]\Bigg\}g^{2,2}(t),
		\\&g^{1,1}(T)=g^{2,1}(T)=g^{1,2}(T)=g^{2,2}(T)=1,
	\end{aligned}
	\right.
\end{equation}
where if $\lambda_1^2+\lambda_2^2\ne0,$
\begin{equation}\label{probability}
  \!\!  \!\!  \!\!   \!\!  \!\!  \!\!  \!\!   \!\!  \!\!  \!\!  \!\!  \!\!  \left\{
	\begin{aligned}
		p(t,1,1)&=\frac{\lambda_2}{\lambda_1+\lambda_2}+ \frac{\lambda_1}{\lambda_1+\lambda_2}e^{-(\lambda_1+\lambda_2)(T-t)},
		\\
		p(t,1,2)&=\frac{\lambda_1}{\lambda_1+\lambda_2}\big[1-e^{-(\lambda_1+\lambda_2)(T-t)}\big],
		\\
		p(t,2,1)&= \frac{\lambda_2}{\lambda_1+\lambda_2}\big[1-e^{-(\lambda_1+\lambda_2)(T-t)}\big],
		\\
		p(t,2,2)&=\frac{\lambda_1}{\lambda_1+\lambda_2}+ \frac{\lambda_2}{\lambda_1+\lambda_2}e^{-(\lambda_1+\lambda_2)(T-t)}.
	\end{aligned}
	\right.
\end{equation}
and if $\lambda_1=\lambda_2=0$, $p(t,1,1)=p(t,2,2)=1,\,p(t,1,2)=p(t,2,1)=0$.
\subsection{\bf The existence of the global solution of ODEs \eqref{odes1}}
 ODEs \eqref{odes1} is a non-linear system, which does not admit explicit solution. In this subsection, we study the global existence and uniqueness of solution to ODEs \eqref{odes1}. We have the following.
\begin{theorem} \label{odesexistence}
ODEs \eqref{odes1} has a unique global solution on $(-\infty,T]$. Besides, the solution \\ $\{g^{i,j}(t)\}_{i,j\in S}$ satisfies
\begin{equation}\label{gconditon}
g^{i,j}\in C^1((-\infty,T]),\;g^{i,j}(t)>0,\;\forall\;t\in(-\infty,T],\,i,j\in S.
\end{equation}
\end{theorem}
To prove Theorem \ref{odesexistence}, we first introduce the following definition and lemmas: 
\begin{definition}
We say $(I,\vec{G})$ is a solution of ODEs \eqref{odes1} if and only if it satisfies the following conditions:
	\\
	(a) $I$ is an interval and $T\in I$;
	\\
	(b) $\vec{G}(t):t\to \left (g^{1,1}(t),g^{2,1}(t),g^{1,2}(t),g^{2,2}(t)\right )^T,\; I\to (\mathbb{R}^+)^4$ is differentiable and $\vec{G}(T)=(1,1,1,1)^T$;
	\\
	(c) $\forall \;t\in I\; ,\vec{G}(t)$ satisfies ODEs \eqref{odes1}.
\end{definition}
The set of the solution of  ODEs \eqref{odes1} is denoted by $\mathscr{G}$. Based on theory on the local existence of solutions to the initial value problem for ODEs (see,  \cite{o1997existence}, Chapter 3, Theorem 3.3), we have $\mathscr{G}\ne \emptyset$.
\begin{lemma}\label{lemma1}
If $(I_1,\vec{G_1})\in\mathscr{G},\;(I_2,\vec{G_2})\in\mathscr{G}$, then
	$$\vec{G_1}(t)=\vec{G_2}(t)\; , \forall \; t\in I_1\cap I_2.$$
\end{lemma}
\begin{proof}
Assume $\exists \;t^*\in I_1\cap I_2$ such that $ \vec{G_1}(t^*)\neq\vec{G_2}(t^*)$. Without loss of generality, let $t^*<T$. Denote
	$$t_*= \sup\left \{t^*<T: \vec{G_1}(t^*)\neq\vec{G_2}(t^*)\right \}\in I_1\cap I_2.$$
Using  the continuity of $\vec{G_1}$ and $\vec{G_2}$, we have  $\vec{G_1}(t_*)=\vec{G_2}(t_*)$. However, 
	$$\vec{G_1}|_{(t_*-\delta,t_*)}\neq\vec{G_2}|_{(t_*-\delta,t_*)},\;\forall\;0<\delta<<1.$$
This contradicts the uniqueness of the local solution to the initial value problem with the initial condition $(t_*,\vec{G_1}(t_*))$. Thus Lemma \ref{lemma1} holds.
\end{proof}
\begin{lemma}\label{lemma2}
Considering the following initial value problem
	\begin{equation}\label{ode}
	\left\{
		\begin{aligned}
		&x^{'}=d(x-a)(x+b),
		\\
		&x(T)=1,
		\end{aligned}
		\right.	    
	\end{equation}
	where $a,b,d,T>0$, then ODE \eqref{ode} exists a global solution on $(-\infty,T]$, and $\forall\;t\in(-\infty,T)$, 
	$$x(t)\in \left\{\begin{aligned}
		(1,a), & \quad a\in(1,+\infty),\\
		1, & \quad a = 1,\\
		(a,1), & \quad a\in(0,1).
	\end{aligned}\right.$$
\end{lemma}
\begin{proof}
The global existence of the initial value problem \eqref{ode} on $(-\infty,T]$ can be obtained directly. 
 Without loss of generality, consider the case where $a>1$. Define $x_1(t)\equiv a,\forall\,t\in(-\infty,T]$. It is straightforward to verify that $x_1$ satisfies  ODE \eqref{ode}. By comparison theorem, $x(t)<a,\forall\,t\in(-\infty,T]$. Next, define $x_2(t)\equiv 1,\forall\,t\in(-\infty,T]$, then $x_2$ satisfies the following initial value problem
        $$
	    \left\{
		\begin{aligned}
		&x^{'}=d(x-a)(x+b)-d(1-a)(1+b),
		\\
		&x(T)=1.
		\end{aligned}
		\right.	    
        $$  
Again by comparison theorem, $x(t)>1,\forall\,t\in(-\infty,T]$. The other cases can be proved by the same way. Therefore Lemma \ref{lemma2} holds.
\end{proof}
Now we present the proof of Theorem \ref{odesexistence} as follows:

\begin{proof}[\bf Proof of Theorem \ref{odesexistence}]
Consider the set \begin{equation}
 C=\left\{c\in(-\infty,T): \ \exists\;\vec{G}\ \mbox{such that}
 \ ((c,T],\vec{G})\in\mathscr{G}\right\}.   
\end{equation}
By the local existence of ODEs \eqref{odes1}, $C\ne\emptyset$. To prove $\inf\left\{C\right\}=-\infty$, we employ a proof by contradiction. Assume $$\inf\{C\}=c^*\in(-\infty,T).$$ Then $$\;\exists \;\{c_n\}_{n\ge1}\downarrow c^*,\;((c_n,T],\vec{G_n})\in \mathscr{G},\;\forall\;n\in \mathbb{N}.$$
Define $$\vec{G^*}(t)=\vec{G_n}(t),\;t>c_n,\;\forall \;t\in(c^*,T].$$ By using Lemma \ref{lemma1}, $\vec{G^*}$ is well defined,  then $((c^*,T],\vec{G^*})\in \mathscr{G}$. Denote $$\vec{G^*}(t)=(g^{1,1}(t),g^{2,1}(t),g^{1,2}(t),g^{2,2}(t))^T,\;t\in(c^*,T].$$
Denote $x(t)=(\frac{g^{1,1}(t)}{g^{2,1}(t)})^{\alpha_1},\ t\in(c^*,T]$, we have
\begin{equation*}
    x^{'}(t)=\lambda_2x^2(t)-\lambda_1+(1-\alpha)M(t)x(t),
\end{equation*}
where 
\begin{equation*}
    \begin{aligned}
    M(t)=&\frac{1}{2}\alpha (\frac{\mu_1-r_1}{\sigma_1})^2 A_1^2(t)-(\frac{\mu_1-r_1}{\sigma_1})^2 A_1(t)-r_1+\frac{\lambda_1}{1-\alpha}\\&-\Big\{\frac{1}{2}\alpha (\frac{\mu_2-r_2}{\sigma_2})^2 A_2^2(t)-(\frac{\mu_2-r_2}{\sigma_2})^2 A_2(t)-r_2+\frac{\lambda_2}{1-\alpha}\Big\}.
    \end{aligned}
\end{equation*}
Then, using $A_1(t),A_2(t)\in[\frac{1}{\alpha_2},\frac{1}{\alpha_1}]$, we see  $M(t)\in[M_1,M_2]$ for some constants $M_1,M_2\in \mathbb{R}$.

Considering the following two initial value problems:
\begin{equation*}
    \left\{
	\begin{aligned}
		&y^{'}=\lambda_2y^2-\lambda_1+(1-\alpha)M_1y,
		\\
		&y(T)=1,
	\end{aligned}
	\right.
\end{equation*}
\begin{equation*}
    \left\{
	\begin{aligned}
		&z^{'}=\lambda_2z^2-\lambda_1+(1-\alpha)M_2z,
		\\
		&z(T)=1.
	\end{aligned}
	\right.
\end{equation*}
For the case where $\alpha\in(0,1)$, based on the comparison theorem, we have  $$\forall\,t\in(c^*,T],z(t)\le x(t)\le y(t), 
  \forall   t\in (c^*,T].$$
Because $\lambda_1,\lambda_2>0$, the equations $\lambda_2y^2-\lambda_1+(1-\alpha)M_1y=0$ and $\lambda_2z^2-\lambda_1+(1-\alpha)M_2z=0$ each have one positive and one negative roots. Therefore, by Lemma \ref{lemma2}, $$\exists\; N_1,N_2>0 \
\mbox{such that} \ y(t)\le N_2,\;z(t)\ge N_1,\;\forall \;t \in (c^*,T] .$$
Hence $x(t)\in[N_1,N_2],\;\forall\;t\in(c^*,T]$.
The same result can be proved in a similar way as in  the case where $\alpha \in(1,\infty)$.
In brief, \begin{equation}\label{neq1}
		\exists N_1,N_2>0\ \mbox{such that}\ \ (\frac{g^{1,1}(t)}{g^{2,1}(t)})^{\alpha_1} \in[N_1,N_2],\;\forall\;t\in(c^*,T].
	\end{equation} 
 And similarly, 
 \begin{equation}\label{neq2}
		\exists N_3,N_4>0  \ \mbox{such that} \  \ (\frac{g^{1,2}(t)}{g^{2,2}(t)})^{\alpha_2} \in[N_3,N_4],\;\forall\;t\in(c^*,T].
	\end{equation}
 Because $g^{i,j}(t)>0\quad\forall\,t\in(c^*,T],i,j\in S$,  ODEs \eqref{odes1} is equivalent to the following: 
 $\forall i,j\in S $, 
 \begin{equation}\label{odes2}
			\left\{
		\begin{aligned}
		&(\log(g^{i,j}))_t(t)=\frac{1-\alpha_j}{\alpha_j}\left\{\frac{1}{2}\alpha_j (\frac{\mu_i\!\!-\!\!r_i}{\sigma_i})^2 A_i^2(t)-(\frac{\mu_i\!\!-\!\!r_i}{\sigma_i})^2 A_i(t)
			-r_i+\frac{\lambda_i}{1\!\!-\!\!\alpha_j}\left[1-(\frac{g^{\hat i,j}(t)}{g^{i,j}(t)})^{\alpha_j}\right]\right\},
   			\\
		&g^{1,1}(T)=g^{2,1}(T)=g^{1,2}(T)=g^{2,2}(T)=1,
		\end{aligned}
		\right.
	\end{equation}
 where $\hat i=3-i$.
 
 Denote ODEs \eqref{odes2} as \begin{equation}\label{odes3}
	    \left\{
	    \begin{aligned}
	&(\log(g^{i,j}))_t(t)=f^{i,j}(t,g^{1,1}(t),g^{2,1}(t),g^{1,2}(t),g^{2,2}(t)), \forall i,j\in S,
		\\ &
		g^{1,1}(T)=g^{2,1}(T) =g^{1,2}(T)=g^{2,2}(T)=1.
	\end{aligned}
	\right.
	\end{equation}
 By using \eqref{neq1}, \eqref{neq2} and the boundedness of $A_1(t)$ and  $A_2(t)$,
 $$\forall\;i,j\in\{1,2\},\;\exists \;L>0\ \mbox{such that } \  |(\log(g^{i,j}))_t(t)|\le L,\;\forall\;t\in(c^*,T].$$
Then  by Lagrange mean value theorem, $(\log(g^{i,j}))(t)$ is Lipschitz-continuous on  $(c^*,T]$, and hence uniformly continuous. Therefore, $\lim_{t\to {c^*}^+}(\log(g^{i,j}))(t)$ exists and is finite. Consequently, $(\log(g^{i,j}))(t)$ can be continuous extended at $t=c^*$. Define $(\log(g^{i,j}))(c^*)=\lim\limits_{t\to {c^*}^+}(\log(g^{i,j}))(t)$, and we have $g^{i,j}(c^*)>0$. By L'Hôpital's rule, \begin{equation}
	    \begin{aligned}
	            (\log(g^{i,j}))_t(c^*)&=\lim_{t\to {c^*}^+}\frac{(\log(g^{i,j}))(t)-(\log(g^{i,j}))(c^*)}{t-c^*}=\lim_{t\to {c^*}^+}(\log(g^{i,j}))_t(t)
	\\
	&=\lim_{t\to {c^*}^+}f^{i,j}(t,g^{1,1}(t),g^{2,1}(t),g^{1,2}(t),g^{2,2}(t))\\
	&=f^{i,j}(c^*,g^{1,1}(c^*),g^{2,1}(c^*),g^{1,2}(c^*),g^{2,2}(c^*)).
	    \end{aligned}
	\end{equation}
 We still denote the function after extension as $\vec{G^*}(t)=(g^{1,1}(t),g^{2,1}(t),g^{1,2}(t),g^{2,2}(t))^T$. It is differentiable and satisfies ODEs \eqref{odes3} on $[c^*,T]$. Because of the equivalence between ODEs \eqref{odes1} and ODEs \eqref{odes3}, $\vec{G^*}(t)$ satisfies ODEs \eqref{odes1}. Therefore  $([c^*,T],\vec{G^*})\in \mathscr{G}$.

 Consider the original initial value problem but with the initial value of $(\!c^*,\vec{G^*}(c^*)\!)$. By the local existence, $\exists\,\delta >0$ such that $ ((c^*-\delta,c^*+\delta),\vec{G_*})\in \mathscr{G}$. Define
$$\vec{G^{**}}(t)=
	\left\{\begin{aligned}
		\vec{G^*}(t) & \quad t\in(c^*,T],\\
		\vec{G_*}(t) & \quad t\in(c^*-\delta,c^*+\delta).
	\end{aligned}\right.$$
	By using  Lemma \ref{lemma1}, $\vec{G^{**}}$ is well defined, and $((c^*-\delta,T],\vec{G^{**}})\in \mathscr{G}$. This contradicts $c^*=\inf C$. Therefore, $\inf C=-\infty$. Then $$\exists \{c_n\}_{n\ge1}\downarrow -\infty \ \ \mbox{ such that}\ \ 
   ((c_n,T],\vec{G_n})\in \mathscr{G},\forall\,n\in \mathbb{N}.$$
 Define $$\vec{G^*}(t)=\vec{G_n}(t),\;t>c_n,\;\forall \;t\in(-\infty,T].$$
 By using Lemma \ref{lemma1}, $\vec{G^{*}}$ is well defined, and hence $((-\infty,T],\vec{G^*})\in \mathscr{G}$. $\vec{G^{*}}$ is the global solution of ODEs \eqref{odes1} on $(-\infty,T]$. The uniqueness follows from Lemma \ref{lemma1}. Then the property \eqref{gconditon} is obvious by ODEs \eqref{odes1}.
\end{proof}
\subsection{\bf Equilibrium Strategy}
Using the separation of variables and  the global existence results for ODEs \eqref{odes1}, we derive the feedback form of the equilibrium strategy as follows:
 \begin{equation}\label{strategy21}
    \hat\pi(t,x,i)=\frac{\mu_i-r_i}{\sigma_i^2}\, x \, \frac{p(t,i,1)(g^{i,1}(t))^{\frac{\alpha_1}{1-\alpha_1}}+p(t,i,2)(g^{i,2}(t))^{\frac{\alpha_2}{1-\alpha_2}}}{\alpha_1 p(t,i,1)(g^{i,1}(t))^{\frac{\alpha_1}{1-\alpha_1}}+\alpha_2 p(t,i,2)(g^{i,2}(t))^{\frac{\alpha_2}{1-\alpha_2}}}.
\end{equation}
And the specific form \eqref{solution1} of $\{f^{i,j}(t,x)\}_{i,j\in S}$ is
$$
    f^{i,j}(t,x)=\frac{1}{1-\alpha_j}x^{1-\alpha_j}(g^{i,j}(t))^{\alpha_j},\;\forall\,i,j\in S.
$$
The following theorem asserts that \eqref{strategy21} is  the equilibrium strategy of the original problem.
\begin{theorem}\label{check}
The family of functions $\{f^{i,j}(t,x)\}_{i,j\in S}$ defined by  the form \eqref{solution1} of separation of variables satisfies  the requirements of the Theorem \ref{verification}. Moreover, $\hat \pi$ given by \eqref{strategy21} is the equilibrium strategy of the original problem.
\end{theorem}
\begin{proof}[\bf Proof of Theorem \ref{check}]
The proof is divided into the following five steps:

{\bf Step 1: $\hat\pi$ is admissible}. Because $\hat\pi$ is linear respect to $x$, the wealth process $X^{\hat\pi}$ satisfies the following linear SDE:
\begin{equation}\label{LSDE}
    d X_t^{\hat\pi}=(\frac{\mu_{\epsilon_t}-r_{\epsilon_t}}{\sigma_{\epsilon_t}})^2 X_t^{\hat\pi}A_{\epsilon_t}(t) d t+r_{\epsilon_t} X_t^{\hat\pi} d t+\frac{\mu_{\epsilon_t}-r_{\epsilon_t}}{\sigma_{\epsilon_t}} X_t^{\hat\pi}A_{\epsilon_t}(t) d W_t.
\end{equation}
Fixing the initial point $X^{\hat\pi}_t=x>0$,  and solving SDE (\ref{LSDE}) (see \cite{karatzas2014brownian}, Chapter 5, Section 6), we have 
\begin{equation}
\begin{aligned}
X^{\hat\pi}_s=x\exp\Bigg\{\int_t^s\Big[(\frac{\mu_{\epsilon_u}-r_{\epsilon_u}}{\sigma_{\epsilon_u}})^2A_{\epsilon_u}(u)+r_{\epsilon_u}-\frac12(\frac{\mu_{\epsilon_u}-r_{\epsilon_u}}{\sigma_{\epsilon_u}})^2 (A_{\epsilon_u}(u))^2 \Big]du\\+\int_t^s\frac{\mu_{\epsilon_u}-r_{\epsilon_u}}{\sigma_{\epsilon_u}}A_{\epsilon_u}(u)d W_u\Bigg\} >0,\;\forall\;s\in[t,T].
\end{aligned}
\end{equation}
Moreover,
$$|{\hat\pi(t,x,i)}|\le\mathop{max}\limits_{i=1,2}\Big\{\frac{\mu_1-r_1}{\sigma_1^2}\Big\}\mathop{max}\limits_{i=1,2}\Big\{\frac{1}{\alpha_1}\Big\}|x|.$$
Therefore, $\hat\pi\in\Pi$ based on Definition \ref{admissible}.

{\bf Step 2: Verify Assumption \ref{assume4}.}
By using Form \eqref{solution1} of separation of variables of $\{f^{i,j}(t,x)\}_{i,j\in S}$, \begin{equation}
    f^{i,j}_{xx}(t,x)=-\alpha_j x^{-\alpha_j-1}(g^{i,j}(t))^{\alpha_j}<0, \;\forall\; (t,x,i,j)\in[0,T)\times \mathbb{R}^+ \times S^2.
\end{equation}
Then Assumption \ref{assume4} holds.

{\bf Step 3: Introduce three useful Lemmas as follows:}
\begin{lemma}\label{lemma3}
$\forall \pi\in\Pi,(t,x,i,j)\in [0,T]\times \mathbb{R}^+ \times S^2$, 
	\begin{equation}\label{estimate1}
	    \mathop{sup}\limits_{s\in[t,T]}\{\mathbb{E}_{t,x,i,j}[|X_s^\pi|^2]\}<\infty.
	\end{equation}
\end{lemma}

\begin{proof}[Proof of Lemma \ref{lemma3}]
Integrating \eqref{wealth} over interval $[s,t]$, we have
	\begin{equation*}
	    X_s^\pi\!=\!X_t^\pi\!+\!\int_t^s\!(\!\mu_{\epsilon_u}\!-\!r_{\epsilon_u}\!)\!\pi\!(\!u,X_u^\pi,\epsilon_u\!)\! d u\!+\!\int_t^s r_{\epsilon_u} X_u^\pi d u\!+\!\int_t^s \pi\!(\!u,X_u^\pi,\epsilon_u\!)\!\sigma_{\epsilon_u} d W_u\, ,
	\forall\,s\in[t,T].
	\end{equation*}
 Using the Cauchy-Schwarz inequality, we obtain
	\begin{equation}\label{calculus2}
		\begin{aligned}
			|X_s^\pi|^2\le& 4\Bigg\{ |X_t^\pi|^2+\Big|\int_t^s(\mu_{\epsilon_u}-r_{\epsilon_u})\pi(u,X_u^\pi,\epsilon_u) d u\Big|^2
			\\&+\Big|\int_t^s r_{\epsilon_u} X_u^\pi d u\Big|^2+\Big|\int_t^s \pi(u,X_u^\pi,\epsilon_u)\sigma_{\epsilon_u} d W_u\Big|^2\Bigg\}
			\\ 		
			\le&4\Bigg\{ |X_t^\pi|^2+(s-t)\int_t^s|(\mu_{\epsilon_u}-r_{\epsilon_u})\pi(u,X_u^\pi,\epsilon_u)|^2 d u
			\\&+(s-t)\int_t^s |r_{\epsilon_u} X_u^\pi|^2 d u+\Big|\int_t^s \pi(u,X_u^\pi,\epsilon_u)\sigma_{\epsilon_u} d W_u\Big|^2\Bigg\}
			\\
			\le&4\Bigg\{ |X_t^\pi|^2+T\int_t^s|(\mu_{\epsilon_u}-r_{\epsilon_u})\pi(u,X_u^\pi,\epsilon_u)|^2 d u
			\\&+T\int_t^s |r_{\epsilon_u} X_u^\pi|^2 d u+\Big|\int_t^s \pi(u,X_u^\pi,\epsilon_u)\sigma_{\epsilon_u} d W_u\Big|^2\Bigg\}.
		\end{aligned}
	\end{equation}	
Denote $C_1=\mathop{max}\limits_{i=1,2}\{(\mu_i-r_i)^2,r_i^2,\sigma_i^2,\lambda_i^2\}$. Using Condition $(ii)$ of Definition \ref{admissible}, we have
\begin{eqnarray}
&& \int_t^s\!|(\mu_{\epsilon_u}\!\!-\!\! r_{\epsilon_u})\pi(u,X_u^\pi,\epsilon_u)|^2 d u\!\!\le \!\!C_1\int_t^s\!|\pi(u,X_u^\pi,\epsilon_u)|^2d u\!\!\le \!\!C_1C^2\int_t^s|X_u^\pi|^2d u,\\ \label{estimate2}
&&\int_t^s |r_{\epsilon_u} X_u^\pi|^2 d u\le C_1\int_t^s|X_u^\pi|^2d u.\label{estimate3}
\end{eqnarray}
Based on  the Burkholder-Davis-Gundy inequality (see \cite{karatzas2014brownian}, Chapter 3, Section 3), there exists a constant $C_2>0$ such that 
\begin{equation}\label{estimate4}
    \begin{aligned}
		\mathbb{E}_{t,x,i,j}\Bigg[\Big|\int_t^s \pi(u,X_u^\pi,\epsilon_u)\sigma_{\epsilon_u} d W_u\Big|^2\Bigg]\le C_2\mathbb{E}_{t,x,i,j}\Bigg[\int_t^s|\pi(u,X_u^\pi,\epsilon_u)\sigma_{\epsilon_u}|^2d u\Bigg]
		\\ \le C_1C_2\mathbb{E}_{t,x,i,j}\Bigg[\int_t^s|\pi(u,X_u^\pi,\epsilon_u)|^2d u\Bigg] \le C_1C_2C^2\mathbb{E}_{t,x,i,j}\Bigg[\int_t^s|X_u^\pi|^2d u\Bigg].
	\end{aligned}
\end{equation}
Taking conditional expectation $\mathbb{E}_{t,x,i,j}[\cdot]$ on \eqref{calculus2}, and  using  \eqref{estimate2}, \eqref{estimate3} and \eqref{estimate4}, 
\begin{equation}
    \mathbb{E}_{t,x,i,j}[|X_s^\pi|^2]\le4\Bigg\{x^2+(TC_1C^2+TC_1+C_1C_2C^2)\int_t^s \mathbb{E}_{t,x,i,j}[|X_u^\pi|^2]d u\Bigg\},\;\forall \;s\in[t,T].
\end{equation}
By Gronwall inequality, $$\mathop{sup}\limits_{s\in[t,T]}\mathbb{E}_{t,x,i,j}[|X_s^\pi|^2]<\infty.$$ 
Therefore, Lemma \ref{lemma3} holds.
\end{proof}
\begin{lemma}\label{lemma4}
	$\forall \pi\in\Pi, \ (t,x,i,j)\in [0,T]\times\mathbb{R}^+ \times S^2, \ p\in[2,+\infty)$, 
	\begin{equation}
	\mathop{sup}\limits_{s\in[t,T]}\{\mathbb{E}_{t,x,i,j}[|X_s^\pi|^{-p}]\}<\infty.
	\end{equation}
\end{lemma}
\begin{proof}[Proof of Lemma \ref{lemma4}]
    By Ito's formula, $\forall\;s\in[t,T]$, 
    \begin{equation}
    \begin{aligned}
		(X_s^\pi)^{-1}=&(X_t^\pi)^{-1}+\int_t^s(-(X_u^\pi)^{-2})(\mu_{\epsilon_u}-r_{\epsilon_u})\pi(u,X_u^\pi,\epsilon_u) d u+\int_t^s(-(X_u^\pi)^{-1}) r_{\epsilon_u} d u
		\\
		&+\int_t^s (-(X_u^\pi)^{-2}) \pi(u,X_u^\pi,\epsilon_u)\sigma_{\epsilon_u} d W_u+\int_t^s (X_u^\pi)^{-3} \pi^2(u,X_u^\pi,\epsilon_u)\sigma_{\epsilon_u}^2 d u.
	\end{aligned}
\end{equation}
As such, using Jensen's inequality, 
\begin{equation}\label{calculus3}
    \begin{aligned}
		|\frac15(X_s^\pi)^{-1}|^p\le&\frac15\Bigg\{|X_t^\pi|^{-p}+\Big|\int_t^s(-(X_u^\pi)^{-2})(\mu_{\epsilon_u}-r_{\epsilon_u})\pi(u,X_u^\pi,\epsilon_u) d u\Big|^p
		\\&+\Big|\int_t^s(-(X_u^\pi)^{-1}) r_{\epsilon_u} d u\Big|^p
		+\Big|\int_t^s (-(X_u^\pi)^{-2}) \pi(u,X_u^\pi,\epsilon_u)\sigma_{\epsilon_u} d W_u\Big|^p
		\\&+\Big|\int_t^s (X_u^\pi)^{-3} \pi^2(u,X_u^\pi,\epsilon_u)\sigma_{\epsilon_u}^2 d u\Big|^p\Bigg\},
		\quad\forall\,s\in[t,T].
	\end{aligned}
\end{equation}
Using Condition $(ii)$ of  Definition\ref{admissible}  and  Hölder's inequality on \eqref{calculus3}, there exists a constant $C_{p,1}=C(p,T,C,C_1)>0$ such that 
\begin{equation}\label{estimate5}
    \begin{aligned}
		&\Big|\int_t^s(-(X_u^\pi)^{-2})(\mu_{\epsilon_u}-r_{\epsilon_u})\pi(u,X_u^\pi,\epsilon_u) d u\Big|^p+\Big|\int_t^s(-(X_u^\pi)^{-1}) r_{\epsilon_u} d u\Big|^p
		\\&+\Big|\int_t^s (X_u^\pi)^{-3} \pi^2(u,X_u^\pi,\epsilon_u)\sigma_{\epsilon_u}^2 d u\Big|^p
		\le C_{p,1}\int_t^s|X_u^\pi|^{-p} d u.
	\end{aligned}
\end{equation}
By the Burkholder-Davis-Gundy inequality, there exists a constant $C_{p,2}>0$ such that 
\begin{equation}\label{estimate6}
    \begin{aligned}
	&\mathbb{E}_{t,x,i,j}\Bigg[\Big|\int_t^s (-(X_u^\pi)^{-2}) \pi(u,X_u^\pi,\epsilon_u)\sigma_{\epsilon_u} d W_u\Big|^p\Bigg]
	\\ &\le C_{p,2}\mathbb{E}_{t,x,i,j}\Bigg[\Big(\int_t^s |(-(X_u^\pi)^{-2}) \pi(u,X_u^\pi,\epsilon_u)\sigma_{\epsilon_u}|^2 d u\Big)^{\frac{p}{2}}\Bigg]
	\\ &\le 
	C_{p,2}C_{p,1}\mathbb{E}_{t,x,i,j}\Bigg[(\int_t^s |X_u^\pi|^{-2} d u)^{\frac{p}{2}}\Bigg]
	\\
	&\le 	C_{p,2}C_{p,1}(s-t)^{\frac{p}{2}-1}\mathbb{E}_{t,x,i,j}\Bigg[\int_t^s |X_u^\pi|^{-p} d u\Bigg].
	\end{aligned}
\end{equation}
Taking conditional expectation $\mathbb{E}_{t,x,i,j}[\cdot]$ on \eqref{calculus3}, and by \eqref{estimate5} and \eqref{estimate6}, there exists a constant $C_{p,3}>0$ such that \begin{equation}
    \mathbb{E}_{t,x,i,j}[|X_s^\pi|^{-p}]\le5^{p-1}x^{-p}+C_{p,3}\int_t^s\mathbb{E}_{t,x,i,j}[|X_u^\pi|^{-p}]d u,\; \forall \;s\in[t,T].
\end{equation}
By Gronwall's inequality, $$\mathop{sup}\limits_{s\in[t,T]}\{\mathbb{E}_{t,x,i,j}[|X_s^\pi|^{-p}]\}<\infty.$$
Then Lemma \ref{lemma4} holds.
\end{proof}
\begin{lemma}\label{lemma5}
$\forall \pi\in\Pi, \ (t,x,i,j)\in [0,T]\times \mathbb{R}^+ \times S^2,\  p\in(-\infty,2]$, 
	
 \begin{equation*}	\mathop{sup}\limits_{s\in[t,T]}\{\mathbb{E}_{t,x,i,j}[|X_s^\pi|^{p}]\}<\infty.
\end{equation*}
\end{lemma}
\begin{proof}[Proof of Lemma \ref{lemma5}]
Using Lemma \ref{lemma3}, Lemma \ref{lemma4} and Hölder's inequality yields that Lemma \ref{lemma5} holds.  
\end{proof}
{\bf Step 4: Verify Condition $(b)$ of Theorem \ref{verification}.} Using Lemma \ref{lemma5}, we have,  $\forall \;(t,x,i,j)\in [0,T]\times \mathbb{R}^+ \times S^2$,
\begin{equation*}
    \begin{aligned}
		&\mathbb{E}_{t,x,i,j}\Bigg[\int_t^T|f_x^{\epsilon_u,j}(t,X^{\hat\pi}_u)\hat\pi(t,X^{\hat\pi}_u,\epsilon_u)
  \sigma_{\epsilon_u}|^2 d u\Bigg]
		=\mathbb{E}_{t,x,i,j}\Bigg[\int_t^T|(X_u^{\hat{\pi}})^{-\alpha_j}(g^{\epsilon_u,j}(t))^{\alpha_j}\hat\pi(t,X^{\hat\pi}_u,\epsilon_u)\sigma_{\epsilon_u}|^2d u\Bigg]
		\\
		&\le C_1C^2C_3\mathbb{E}_{t,x,i,j}\Bigg[\int_t^T|X^{\hat\pi}_u|^{2(1-\alpha_j)}d u\Bigg]
				\le C_1C^2C_3(T-t)\Big\{\mathop{sup}\limits_{s\in[t,T]}\mathbb{E}_{t,x,i,j}[|X_s^{\hat{\pi}}|^{2(1-\alpha_j)}]\Big\}<\infty,
	\end{aligned}
\end{equation*}
where $C_3>0$ is a constant only dependent on $g^{i,j},\alpha_j,i,j\in S$. 

Because $\{W_t\}_{t\in[t,T]}$ is a standard Brownian motion under the conditional probability $\mathbb{P}_{t,x,i,j}$,\\ $\Big\{\int^s_t f_x^{\epsilon_u,j}((t,X^{\hat\pi}_u)\hat\pi(t,X^{\hat\pi}_u,\epsilon_u)\sigma_{\epsilon_u}d W_u\Big\}_{s\in[t,T]}$ is a martingale under the conditional probability $\mathbb{P}_{t,x,i,j}$. Thus, Condition $(b)$ holds.

{\bf Step 5: Verify Condition $(c)$ of Theorem \ref{verification}, i.e., Assumptions \ref{assume1}-\ref{assume3} hold.} Using the same way as in Step 4, we get
$$\mathbb{E}_{t,x,i,j}\left[\int_t^T\!\!|f_x^{\epsilon_u,j}((t,X^{\pi}_u)\pi(t,X^{\pi}_u,\epsilon_u)\sigma_{\epsilon_u}|^2 d u\right]\!\!<\!\!\infty, \forall \pi\in\Pi,(t,x,i,j)\in [0,T]\times \mathbb{R}^+ \times S^2.$$ As such, $\Big\{\int^s_t f_x^{\epsilon_u,j}(u,X_u^\pi)\pi(u,X_u^\pi,\epsilon_u)\sigma_{\epsilon_u}d W_u\Big\}_{s\in[t,T]}$ is a martingale under the conditional probability $\mathbb{P}_{t,x,i,j}$. Thus, Assumption \ref{assume1} holds. 

We have the following estimations: 
\begin{equation*}
\left\{
    \begin{aligned}
            &\mathbb{E}_{t,x,i,j}[|f_t^{{\epsilon_u},j}(u,X_u^\pi)|^2]=\mathbb{E}_{t,x,i,j}[|\frac{\alpha_j}{1-\alpha_j}(X_u^\pi)^{1-\alpha_j}g_t^{\epsilon_u,j}(t)(g^{\epsilon_u,j}(t))^{\alpha_j-1}|^2]
		\le C_3\mathbb{E}_{t,x,i,j}[|X_u^\pi|^{2(1-\alpha_j)}],
		\\
		&\mathbb{E}_{t,x,i,j}[|\frac{1}{2}f^{{\epsilon_u},j}_{xx}(u,X_u^\pi)\sigma_{\epsilon_u}^2\pi^2(u,X_u^\pi,{\epsilon_u})|^2]\le  C_1^2C_3^2C^4\mathbb{E}_{t,x,i,j}[|X_u^\pi|^{2(1-\alpha_j)}],
		\\
		&\mathbb{E}_{t,x,i,j}[|f^{{\epsilon_u},j}_x(u,X_u^\pi) (\mu_{\epsilon_u}-r_{\epsilon_u})\pi(u,X_u^\pi,{\epsilon_u})|^2]\le C_1C_3C^2\mathbb{E}_{t,x,i,j}[|X_u^\pi|^{2(1-\alpha_j)}],
		\\
		&\mathbb{E}_{t,x,i,j}[|f^{{\epsilon_u},j}_x(u,X_u^\pi) X_u^\pi r_{\epsilon_u}|^2]\le C_1C_3\mathbb{E}_{t,x,i,j}[|X_u^\pi|^{2(1-\alpha_j)}],
		\\
		&\mathbb{E}_{t,x,i,j}[|\lambda_{\epsilon_u}f^{{\epsilon_u},j}(u,X_u^\pi)|^2]\le C_1C_3\mathbb{E}_{t,x,i,j}[|X_u^\pi|^{2(1-\alpha_j)}],
		\\
		&\mathbb{E}_{t,x,i,j}[|\lambda_{\epsilon_u}f^{\hat{\epsilon}_u,j}(u,X_u^\pi)|^2]\le C_1C_3\mathbb{E}_{t,x,i,j}[|X_u^\pi|^{2(1-\alpha_j)}].
    \end{aligned}\right.
\end{equation*}
Based on  the definition of generator $\mathcal{A}^\pi$, we also have
\begin{equation*}
    	\begin{aligned}
		\mathcal{A}^\pi f^{{\epsilon_u},j}(u,X_u^\pi)
		=&f_t^{{\epsilon_u},j}(u,X_u^\pi)+\frac{1}{2}f^{{\epsilon_u},j}_{xx}(u,X_u^\pi)\sigma_{\epsilon_u}^2\pi^2(u,X_u^\pi,{\epsilon_u})
		\\&+f^{{\epsilon_u},j}_x(u,X_u^\pi) (\mu_{\epsilon_u}-r_{\epsilon_u})\pi(u,X_u^\pi,{\epsilon_u})
		+f^{{\epsilon_u},j}_x(u,X_u^\pi) X_u^\pi r_{\epsilon_u}
		\\&-\lambda_{\epsilon_u}f^{{\epsilon_u},j}(u,X_u^\pi)+\lambda_{\epsilon_u}f^{\hat{\epsilon}_u,j}(u,X_u^\pi),\; (\hat{\epsilon}_u=3-{\epsilon_u}).
	\end{aligned}
\end{equation*}
By Cauchy-Schwarz inequality, 
\begin{equation*}
	    \mathbb{E}_{t,x,i,j}[|\mathcal{A}^\pi f^{\epsilon_u,j}(u,X_u^\pi)|^2]\le C_4\mathbb{E}_{t,x,i,j}[|X_u^\pi|^{2(1-\alpha_j)}],
	\end{equation*}
where $C_4>0$ is a constant only depending on $C,\ C_1,\ C_3$. Additionally, using Lemma \ref{lemma5},
\begin{equation*}
        \mathop{sup}\limits_{s\in[t,T]}\{\mathbb{E}_{t,x,i,j}[|\mathcal{A}^\pi f^{\epsilon_u,j}(u,X_u^\pi)|^2]\}<\infty.
\end{equation*}
Then 
\begin{equation*}
       \begin{aligned}
		&\mathop{sup}\limits_{h\in(0,T\!-\!t]}\!\mathbb{E}_{t,x,i,j}\!\Bigg[\!\Big|\!\frac{1}{h} \!\int^{t\!+\!h}_t\! \mathcal{A}^\pi\! f^{\epsilon_u,j}\!(\!u,\!X_u^\pi)d u\Big|^2\!\Bigg]\!
		\!\le\!
		\mathop{sup}\limits_{h\in(0,T\!-\!t]}\!\mathbb{E}_{t,x,i,j}\!\Bigg[\!\frac{1}{h} \!\int^{t\!+\!h}_t\! |\mathcal{A}^\pi\! f^{\epsilon_u,j}\!(\!u,\!X_u^\pi)|^2d u\!\Bigg]\!
		\\
  &\le
		\mathop{sup}\limits_{s\in[t,T]}\mathbb{E}_{t,x,i,j}[|\mathcal{A}^\pi f^{\epsilon_u,j}(u,X_u^\pi)|^2]<\infty.
	\end{aligned} 
\end{equation*}
As such, $\Big\{\frac{1}{h} \int^{t+h}_t \mathcal{A}^\pi f^{\epsilon_u,j}(u,X_u^\pi)d u\Big\}_{h\in(0,T-t]}$ is $L^2$-bounded under $\mathbb{P}_{t,x,i,j}$, and thus uniformly integrable. Assumption \ref{assume2} holds.

Because we have just proved that $\Big\{\int^s_t f_x^{\epsilon_u,j}(u,X_u^\pi)\pi(u,X_u^\pi,\epsilon_u)\sigma_{\epsilon_u}d W_u\Big\}_{s\in[t,T]}$ is a martingale under $\mathbb{P}_{t,x,i,j}$, integrating \eqref{ito2} over interval $(t,t+h)$ and taking conditional expectation $\mathbb{E}_{t,x,i,j}[\cdot]$, we have 
\begin{equation*}
	    \mathbb{E}_{t,x,i,j}[f^{\epsilon_{t+h},j}(t+h,X^\pi(t+h))]-f^{i,j}(t,x)= \mathbb{E}_{t,x,i,j}\Bigg[\int^{t+h}_t \mathcal{A}^\pi f^{\epsilon_u,j}(u,X_u^\pi)d u\Bigg].
\end{equation*}
Using Assumption \ref{assume2} holds, we obtain
\begin{equation*}
        	\mathop{lim}\limits_{h\to 0^+}\mathbb{E}_{t,x,i,j}\Bigg[\frac{1}{h} \int^{t+h}_t \mathcal{A}^\pi f^{\epsilon_u,j}(u,X_u^\pi)d u\Bigg]=\mathcal{A}^\pi f^{i,j}(t,x).
    \end{equation*}
Then
\begin{equation*} 
        	\mathop{lim}\limits_{h\to 0^+}\mathbb{E}_{t,x,i,j}\Big[ \int^{t+h}_t \mathcal{A}_{\epsilon_u}^\pi f^{\epsilon_u,j}(u,X_u^\pi)d u\Big]=0.
    \end{equation*}
Hence
\begin{equation*}
	    \mathop{lim}\limits_{h\to 0^+}\mathbb{E}_{t,x,i,j}[f^{\epsilon_{t+h},j}(t+h,X^\pi(t+h))]=f^{i,j}(t,x).
	\end{equation*}
i.e.,  Assumption \ref{assume3} holds. Thus we complete the proof of Theorem \ref{check}.
\end{proof}

Next, we consider two special cases.
\begin{corollary}
When the investor exhibits the same risk aversion in both the bull and bear regimes, i.e., \(\alpha_1 = \alpha_2 = \alpha\), the equilibrium strategy from \eqref{strategy2} is given by
\begin{equation}\label{equ:pi1}
    \hat{\pi}(t,x,i) = \frac{\mu_i - r_i}{\alpha \sigma_i^2}\,x.
\end{equation}
\end{corollary}

It is clear that when the utility function is independent of beliefs, i.e., \(u^j(x) = u(x) = \frac{1}{1-\alpha}x^{1-\alpha}\), we have
\[
\sum_{j \in S} p(\!t,\!i,\!j) u^{-1}\left(\mathbb{E}_{t,\!x,\!i,\!j}[u(X^\pi_T)]\right) 
\neq u^{-1}\left(\sum_{j \in S} p(\!t,\!i,\!j) \mathbb{E}_{t,\!x,\!i,\!j}[u(X^\pi_T)]\right) = u^{-1} \left(\mathbb{E}_{t,\!x,\!i}[u(X^\pi_T)]\right),
\]
which indicates that even when \(u^j(x)\) is beliefs-independent, the objective function \eqref{objective function} is not equivalent to the standard expected utility maximization problem.

However, the strategy \eqref{equ:pi1} matches the optimal investment strategy derived in \cite{sotomayor2009explicit}, meaning that the equilibrium strategy in this case coincides with the optimal strategy of the expected utility maximization problem under regime-switching. The reason these different objectives yield the same result is that \eqref{equ:pi1} is optimal for the following problem:
\begin{equation*}
    \max_\pi \;\mathbb{E}_{t,x,i,j}\left[u\left(X^\pi_T\right)\right].
\end{equation*}
As \(u^{-1}\) is an increasing function, it follows directly that \eqref{equ:pi1} maximizes both \eqref{objective function} and the expected utility maximization problem under regime-switching.

\begin{corollary}\label{cor:2}
When the market coefficients are independent of \(i\) and \(\epsilon_T\) follows a Bernoulli distribution, i.e., \(\mathbb{P}(\epsilon_T = 1) = p\) and \(\mathbb{P}(\epsilon_T = 2) = q = 1 - p\), the equilibrium investment strategy is given by
\begin{equation*}\label{equ:pi2}
    \hat{\pi}(t,x) = \frac{\mu - r}{\sigma^2} \, x \, \frac{p \left(g^{1}(t)\right)^{\frac{\alpha_1}{1 - \alpha_1}} + q \left(g^{2}(t)\right)^{\frac{\alpha_2}{1 - \alpha_2}}}{\alpha_1 p \left(g^{1}(t)\right)^{\frac{\alpha_1}{1 - \alpha_1}} + \alpha_2 q \left(g^{2}(t)\right)^{\frac{\alpha_2}{1 - \alpha_2}}},
\end{equation*}
where \(g^{1}(t)\) and \(g^{2}(t)\) satisfy the following ODEs 
\begin{equation*}
    \left\{
	\begin{aligned}
		\frac{\alpha_1}{1-\alpha_1}g_t^{1}(t)=&\Bigg\{\frac{1}{2}\alpha_1 (\frac{\mu-r}{\sigma})^2 A^2(t)-(\frac{\mu-r}{\sigma})^2 A(t)
		-r\Big]\Bigg\}g^{1}(t),
		\\
		\frac{\alpha}{1-\alpha}g_t^{2}(t)=&\Bigg\{\frac{1}{2}\alpha (\frac{\mu-r}{\sigma})^2 A^2(t)-(\frac{\mu-r}{\sigma})^2 A(t)
		-r\Big]\Bigg\}g^{2}(t), 
		\\g^{1}(T)=g^{2}(T)&=1,
	\end{aligned}
	\right.
\end{equation*}
where
\begin{equation*}
A(t)\mathop{=}\limits^{\Delta}\frac{p(g^{1}(t))^{\frac{\alpha_1}{1-\alpha_1}}+q(g^{2}(t))^{\frac{\alpha_2}{1-\alpha_2}}}{\alpha_1 p(g^{1}(t))^{\frac{\alpha_1}{1-\alpha_1}}+\alpha_2 q(g^{2}(t))^{\frac{\alpha_2}{1-\alpha_2}}}.
\end{equation*}
\end{corollary}
In this case, the financial model aligns with that of \cite{desmettre2023equilibrium}, and as shown in Corollary \eqref{cor:2}, the results are consistent with those found in \cite{desmettre2023equilibrium}.

\vskip 15pt
\section{\bf Economic Analysis and Numerical Examples}\label{sec5}
\subsection{\bf The condition of market coefficients} Following  \cite{fama1989business} and \cite{french1987expected}, we assume  $r_i\le \mu_i,\,\forall\,i\in S$. Moreover, we observe that the feedback form of equilibrium strategy \eqref{strategy2} is directly proportional to the excess return of the stock $(\mu_i-r_i)$ and inversely proportional to the variance of the stock $\sigma_i^2$ under the given regime. Therefore, in order to study the behavior of equilibrium strategy, it is necessary to investigate the relationship between $(\mu_i-r_i)$ and $\sigma_i^2$ under different regimes. Using empirical data, \cite{fama1989business} show that positive expected excess returns $(\mu_i-r_i)$ are lower during a bull market and higher during a bear market. The empirical study of \cite{schwert1989does} demonstrates that  the stock volatility is higher during a bear market than a bull market. The study by \cite{french1987expected} further confirms the positive correlation between the excess return of stocks and the volatility of stocks returns. Additionally, it is noted that the ratio $\frac{\mu_i-r_i}{\sigma^2_i}$ is bigger during strong market conditions. Therefore, we assume that the market coefficients satisfy the following conditions: 
\begin{equation}\label{condition}
    0<\mu_1-r_1<\mu_2-r_2,\;0<\sigma_1<\sigma_2,\;\frac{\mu_1-r_1}{\sigma_1^2}>\frac{\mu_2-r_2}{\sigma_2^2}.
\end{equation}
\subsection{\bf The property of equilibrium strategy} Let $\pi^*$ denote the proportion of investment in risky asset when investor adapts equilibrium strategy. Then we have \begin{equation}\label{strategy4}
    \begin{aligned}
    \pi^*(t,x,i)=\frac{\hat\pi(t,x,i)}{x}=&\frac{\mu_i-r_i}{\sigma_i^2}\frac{p(t,i,1)(g^{i,1}(t))^{\frac{\alpha_1}{1-\alpha_1}}+p(t,i,2)((g^{i,2}(t))^{\frac{\alpha_2}{1-\alpha_2}}}{\alpha_1 p(t,i,1)(g^{i,1}(t))^{\frac{\alpha_1}{1-\alpha_1}}+\alpha_2 p(t,i,2)((g^{i,2}(t))^{\frac{\alpha_2}{1-\alpha_2}}}
    \\=&\frac{\mu_i-r_i}{\sigma_i^2}A_i(t).
    \end{aligned}
\end{equation}
Compared to the expected utility maximization under regime-switching (see \cite{sotomayor2009explicit}), we see that the equilibrium investment is a Merton's fraction multiplied by a deterministic function. 
\begin{itemize}
    \item Without loss of generality, we continue to denote $\pi^*$ as the  equilibrium strategy. \eqref{strategy4} indicates that $\pi^*(t,x,i)$ is independent of the investor's current wealth $x$. Thus  $\pi^*(t,x,i)=\pi^*(t,i)$. Although we have the coefficients condition $\frac{\mu_1-r_1}{\sigma_1^2}>\frac{\mu_2-r_2}{\sigma_2^2}$, unfortunately, it is not straightforward to determine the relative magnitudes of $A_i(t),i\in S$. Consequently, it is not simple to infer the proportion of investment in the risky asset under the equilibrium strategy during bull and bear markets. This comparison will be further explored in the subsequent numerical analysis.
\end{itemize}

A key innovation of the model presented in this paper is that the investor's utility function depends on the belief of regime at the terminal time. In other words, investors can adjust their risk aversion in real-time according to the current regime. In contrast, the study by \cite{rieder2005portfolio} assumes that  a utility function that is constant and independent of the regime, leading to a constant level of risk aversion for investors. In fact, a series of empirical studies, such as \cite{gordon2000preference}, \cite{guiso2018time} and \cite{berrada2018asset}, demonstrate that a low (high) level of risk aversion is associated with bull (bear) regime, which manifest the rationality of our model. Next, we analyze the difference in equilibrium strategies between adjusting risk aversion based on regime and constant risk aversion. Consider two different coefficients of relative risk aversion $\alpha<\beta,\alpha,\beta\in\mathbb{R}^+\backslash\{1\}$. 

Investor A adopts a constant coefficient of relative risk aversion $\alpha$. Letting $\alpha_1=\alpha_2=\alpha$ in \eqref{strategy4}, we have \begin{equation}\label{strategy5}
    \pi^*_1(t,i)=\frac{\mu_i-r_i}{\alpha\sigma_i^2},
\end{equation} 
which is the Merton's fraction.

Investor B adopts a constant coefficient of relative risk aversion $\beta$. Letting $\alpha_1=\alpha_2=\beta$ in \eqref{strategy4}, we have 
\begin{equation}\label{strategy6}
    \pi^*_2(t,i)=\frac{\mu_i-r_i}{\beta\sigma_i^2}.
\end{equation}

Investor C adopts coefficients of relative risk aversion $\alpha$ and $\beta$ in bull and bear markets, respectively. Letting $\alpha_1=\alpha, \alpha_2=\beta$ in \eqref{strategy4}, we have 
\begin{equation}\label{strategy7}
    \pi^*(t,i)=\frac{\mu_i-r_i}{\sigma_i^2}\frac{p(t,i,1)(g^{i,1}(t))^{\frac{\alpha}{1-\alpha}}+p(t,i,2)((g^{i,2}(t))^{\frac{\beta}{1-\beta}}}{\alpha p(t,i,1)(g^{i,1}(t))^{\frac{\alpha}{1-\alpha}}+\beta p(t,i,2)((g^{i,2}(t))^{\frac{\beta}{1-\beta}}}.
\end{equation}
Obviously,
\begin{equation}\label{character1}
    \pi^*_2(t,i)<\pi^*(t,i)<\pi^*_1(t,i),\;\forall\; (t,i)\in [0,T)\times S.
\end{equation}

We begin by considering the case of bull market. On one hand, Investor C invests strictly more in risky assets than Investor B. This aligns with the intuitive notion that Investor C has lower risk aversion in a bull market compared to Investor B. On the other hand, Investor C's proportion of investment in risky assets is strictly less than that of Investor A, despite both having the same risk aversion in a bull market.  This result can only be attributed to the different levels of risk aversion between Investor C and Investor A during bear markets. Intuitively, Investor C's proportion of investment in risky assets in a bull market is also influenced by the possibility of the market transitioning to a bear state, where risk aversion increases. This aligns with the obtained results. When the market is in a bear state, the conclusions are entirely symmetrical and will not be elaborated further.

Additionally, let $t\to T^-$ in \eqref{strategy7}, we obtain \begin{equation}\label{character2}
    \lim_{t\to T^-}\pi^*(t,1)=\pi^*_1(t,1), \lim_{t\to T^-}\pi^*(t,2)=\pi^*_1(t,2).
\end{equation}
This can be interpreted as follows: when the market is in a bull state and approaching the terminal time, the investor C adopting the equilibrium strategy tends to assume that the regime will remain unchanged. Consequently, they disregard the potential increase in risk aversion due to a possible shift to a bear market. This results in the equilibrium investment proportion of investor C converging towards that of investor A, who maintains constant risk aversion. The conclusions are entirely symmetrical when the market is in a bear state.
\begin{remark}
    The equilibrium strategies in \eqref{strategy5} and \eqref{strategy6} are independent of current wealth $x$ and time $t$, only depending on regime $i$. This result is consistent with the optimal portfolio problem with constant market coefficients and coefficient of relative risk aversion (By dynamic programming principle and solving HJB equation, the optimal portfolio can be obtained; see \cite{yong2012stochastic}, Chapter 4). 
\end{remark}
\subsection{\bf Numerical analysis} The parameters are as follows: $\mu_1=0.15$, $r_1=0.05$, $\sigma_1=0.25$, $\mu_2=0.25$, $r_2=0.01$, $\sigma_2=0.6$, which satisfies Condition \eqref{condition}. Considering the complexity of $A_i(t)$ in $\pi^*$, we present the equilibrium strategy $\pi^*$ in the form of curve graphs.
\subsubsection{\bf Effects of $t$} Taking $\alpha =2$, $\beta=3$,  $\lambda_1=1$, $\lambda_2=1$, $T=10$, we have following figure:
\begin{figure}[h]
  \centering
  \subcaptionbox{\label{fig1}}
    {\includegraphics[width=0.45\linewidth]{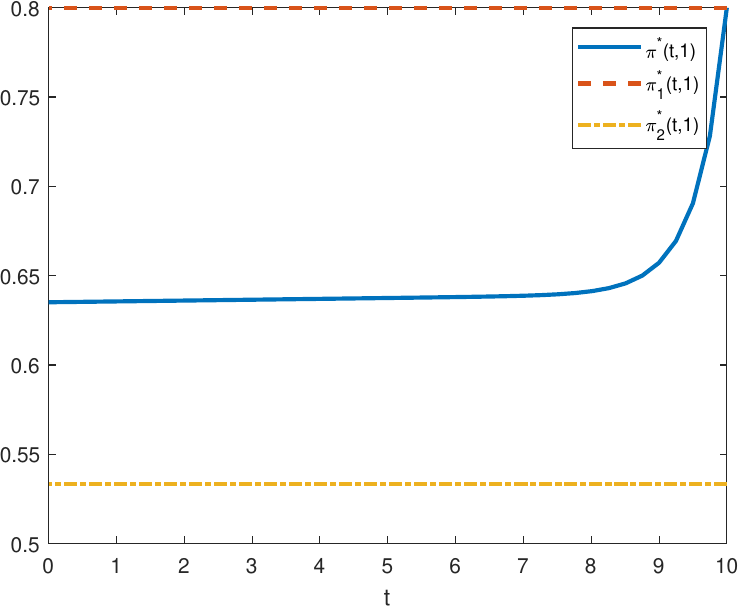}}
  \subcaptionbox{\label{fig2}}
    {\includegraphics[width=0.45\linewidth]{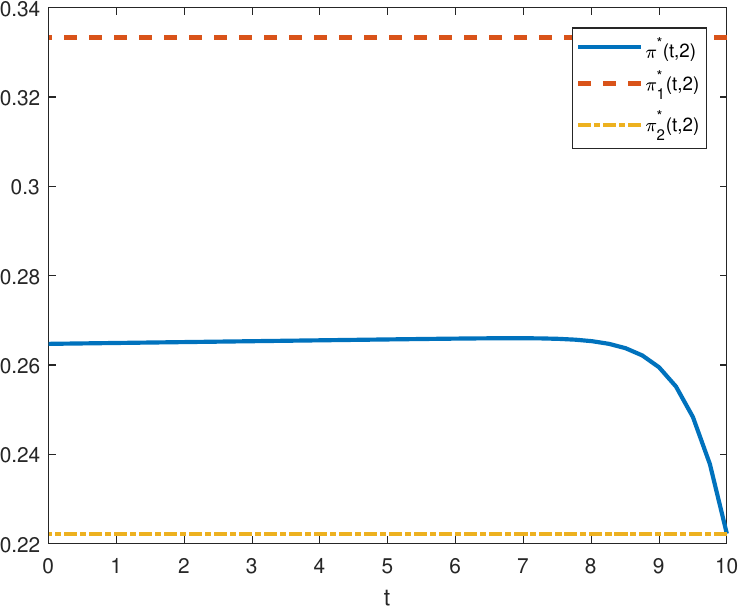}}
  \caption{Effects of $t$.}
  \label{Fig0}
\end{figure}
From Figure \ref{Fig0}, properties \eqref{character1} and \eqref{character2} are evident. Figure \ref{Fig0}\subref{fig1} indicates that the equilibrium investment proportion during a bull market remains relatively stable for $t<8$, but rises rapidly after $t=8$, reaching $\pi^*_1(t,1)=0.8$. Figure \ref{Fig0}\subref{fig2} indicates a similar result. The stability of the investment proportion for $t<8$ can be considered a balance achieved under the joint influence of the current regime and the potential changes in regime. In fact, by \eqref{probability}, we have \begin{equation}\label{limit}
    \lim_{t\to -\infty}p(t,i,j)=1-\frac{\lambda_j}{\lambda_1+\lambda_2},\;\forall\;i,j\in S.
\end{equation} This implies that at earlier times $t$, the probability distribution of the regime at the terminal time is relatively stable, meaning that the distribution of investors' different risk aversions is also stable. As time progresses, the probability that the current regime matches the terminal regime increases and tends towards 1, i.e, $p(t,i,i)\uparrow1$, as  $t\to T^-$, $\forall \;i\in S$. Rationally, investors become increasingly inclined to invest assuming no change in regime, disrupting the balance for $t<8$. This interprets the rapid changes in equilibrium investment for $t>8$.
\subsubsection{\bf Effects of $\lambda$} Taking $\alpha =2,\beta=3,T=10$, varying the values of $\lambda_1$ and $\lambda_2$, we obtain following figures: 
\begin{figure}[h]
  \centering
  \subcaptionbox{$\pi^*(t,1)$\label{fig3}}
    {\includegraphics[width=0.45\linewidth]{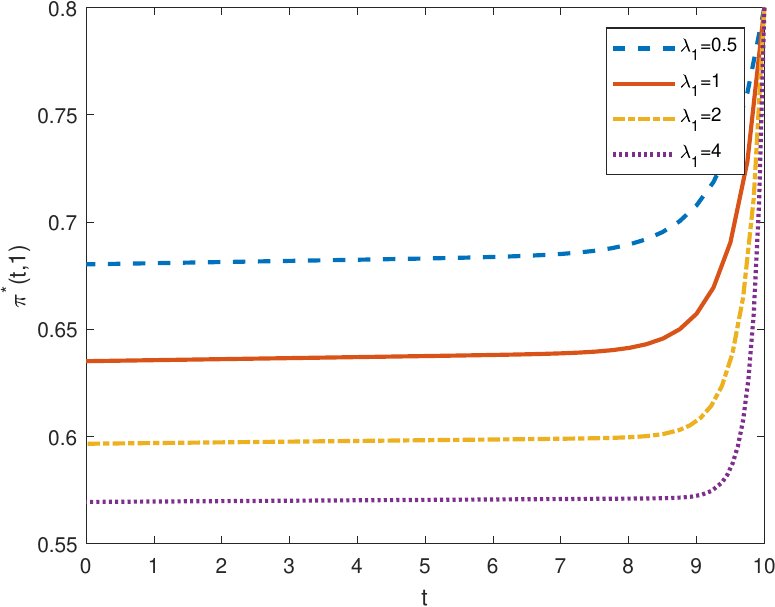}}
  \subcaptionbox{$\pi^*(t,2)$\label{fig4}}
    {\includegraphics[width=0.45\linewidth]{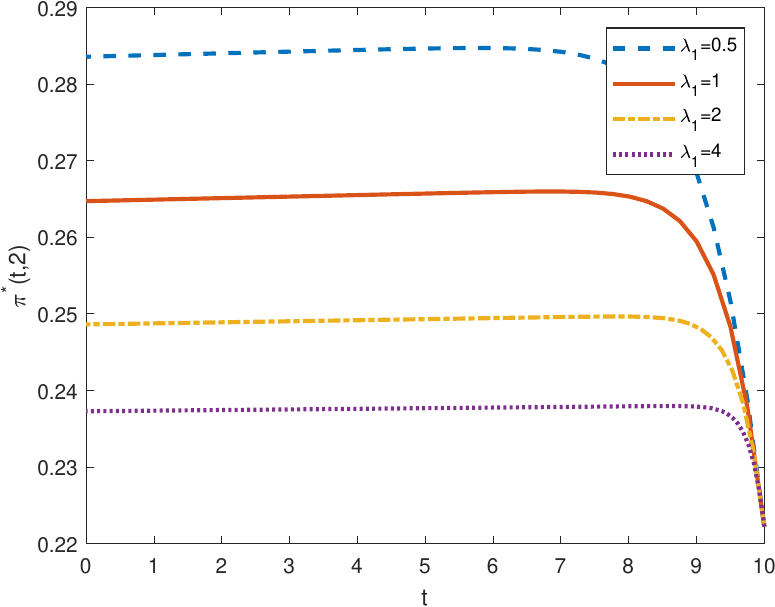}}
  \caption{Effects of $\lambda_1$ ($\lambda_2=1$).}
  \label{Fig1}
\end{figure}
\begin{figure}[h]
  \centering
  \subcaptionbox{$\pi^*(t,1)$\label{fig5}}
    {\includegraphics[width=0.45\linewidth]{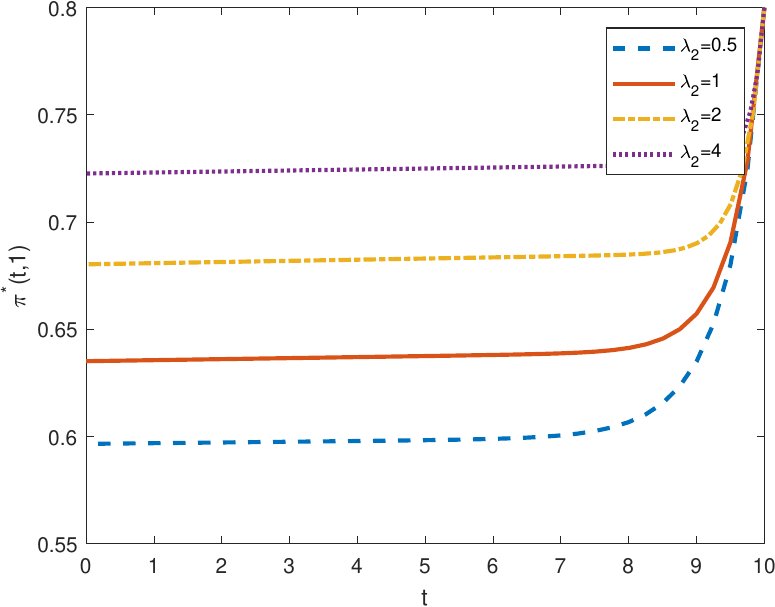}}
  \subcaptionbox{$\pi^*(t,2)$\label{fig6}}
    {\includegraphics[width=0.45\linewidth]{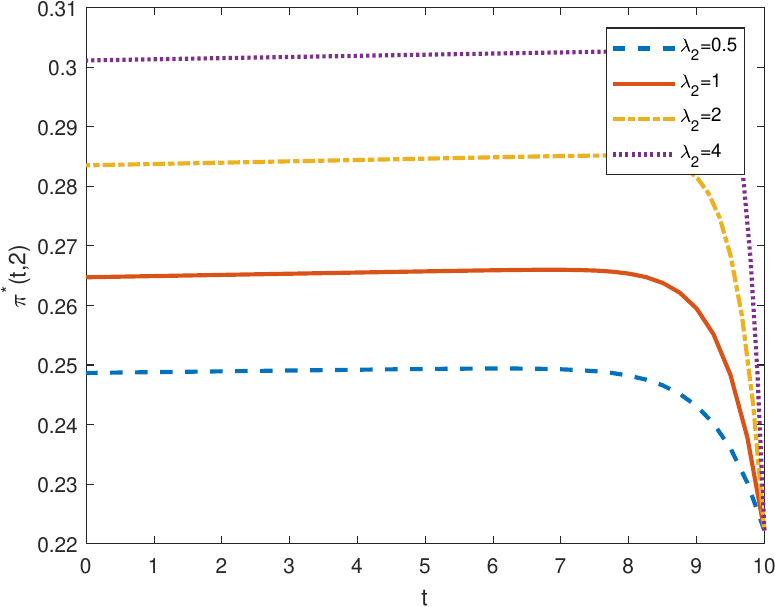}}
  \caption{Effects of $\lambda_2 $ ($\lambda_1=1$).}
  \label{Fig2}
\end{figure}
Figures \ref{Fig1} and \ref{Fig2} show that as the ratio $\frac{\lambda_2}{\lambda_1}$ increases, the equilibrium investment proportion in risky assets rises, irrespective of whether the regime is a bull or bear market. From calculations, the stationary distribution of Markov process $\{\epsilon_t\}_{t\in[0,T]}$ is given by $\widetilde\pi(1)=\frac{\lambda_2}{\lambda_1+\lambda_2},\widetilde\pi(2)=\frac{\lambda_1}{\lambda_1+\lambda_2}$. Consequently, as the ratio $\frac{\lambda_2}{\lambda_1}$ increases, the probability of the regime being bull market becomes higher. Given that investors exhibit lower risk aversion in bull markets compared to bear markets, it is reasonable that the equilibrium investment proportion in risky assets rises as $\frac{\lambda_2}{\lambda_1}$ increases.
\begin{figure}[h]
  \centering
  \subcaptionbox{$\pi^*(t,1)$\label{fig7}}
    {\includegraphics[width=0.45\linewidth]{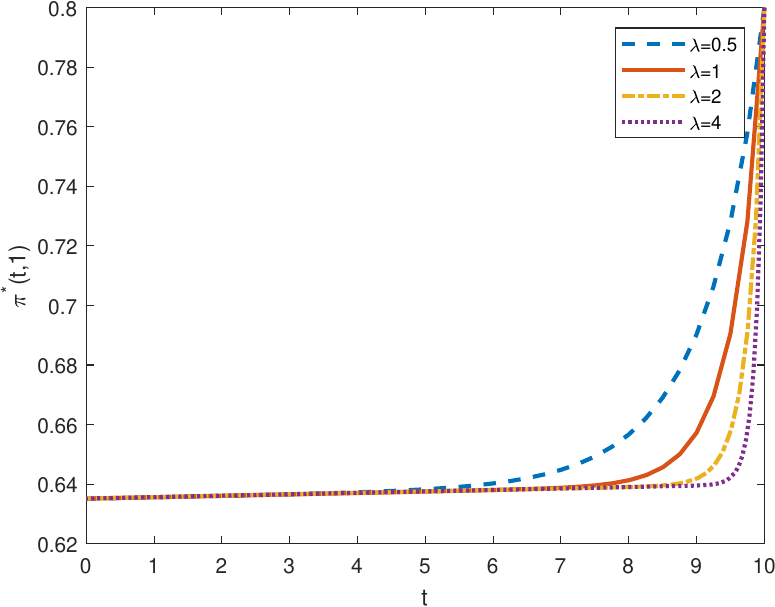}}
  \subcaptionbox{$\pi^*(t,2)$\label{fig8}}
    {\includegraphics[width=0.45\linewidth]{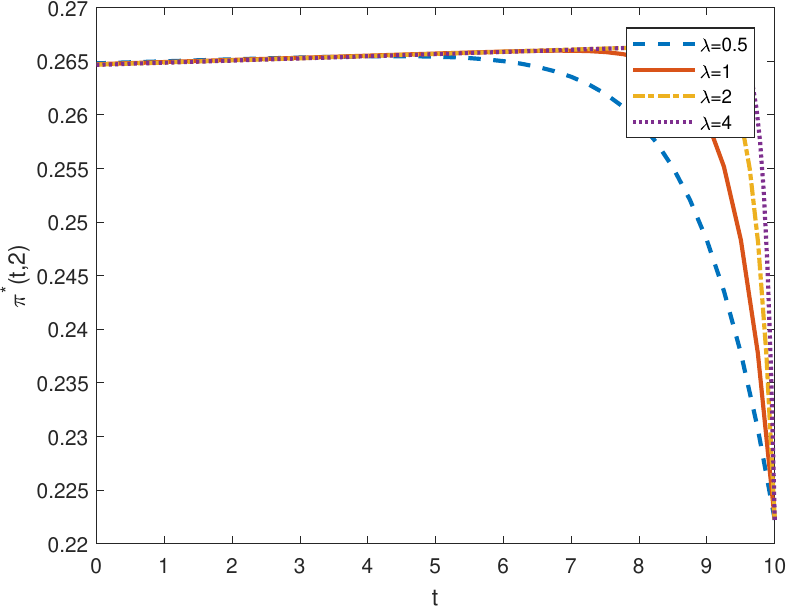}}
  \caption{Effects of $\lambda$ ($\lambda_1=\lambda_2=\lambda$).}
  \label{Fig3}
\end{figure}

In Figure \ref{Fig3}, with the condition $\lambda_1=\lambda_2$, it is noted that the stable portion of the plotted curve almost coincides. Considering Figures \ref{Fig1} and \ref{Fig2}, it is evident that the stable portion of the equilibrium investment is influenced by the ratio $\frac{\lambda_2}{\lambda_1}$, rather than independent $\lambda_1$ or $\lambda_2$. Furthermore, Figure \ref{Fig3} illustrates that the larger the objective of $\lambda$, the longer the stable portion in the first half of the equilibrium investment, and the steeper the slope in the second half. In fact, as $\lambda$ increases, the speed at which $p(t,i,i)\uparrow1\quad(t\to T^-)$ accelerates. Consequently, when $t$ is close to $T$, investors' tendency to invest according to the current regime increases more rapidly. This results in a steeper slope in the second half of the equilibrium investment curve, as shown in the Figure \ref{Fig3}. From another perspective, the expected time $\mathbb{E}[S_i]$ that the regime remains in state $i$ is $\frac{1}{\lambda_i}$. Thus, a larger $\lambda$ implies more frequent regime-switching. As a result, investors are more inclined to consider not only the current market state but also potential future regime-switching in their investment decisions. Consequently, as $\lambda$ increases, the stable portion of the equilibrium investment in the first half becomes longer.
\subsubsection{\bf Effects of risk aversion $\alpha,\beta$} Taking $\lambda_1=\lambda_2=1,T=10$, varying the values of $\alpha$ and $\beta$, we obtain following figures: 
\begin{figure}[h]
  \centering
  \subcaptionbox{$\pi^*(t,1)$\label{fig9}}
    {\includegraphics[width=0.45\linewidth]{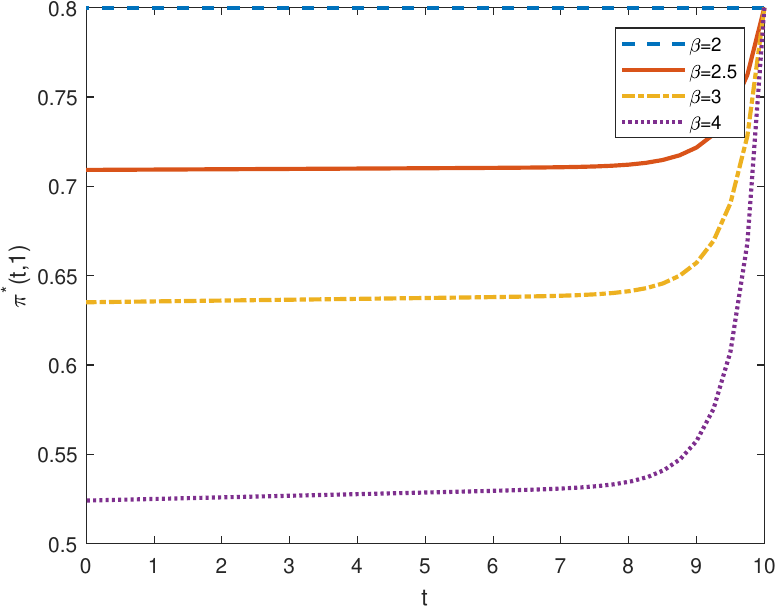}}
  \subcaptionbox{$\pi^*(t,2)$\label{fig10}}
    {\includegraphics[width=0.45\linewidth]{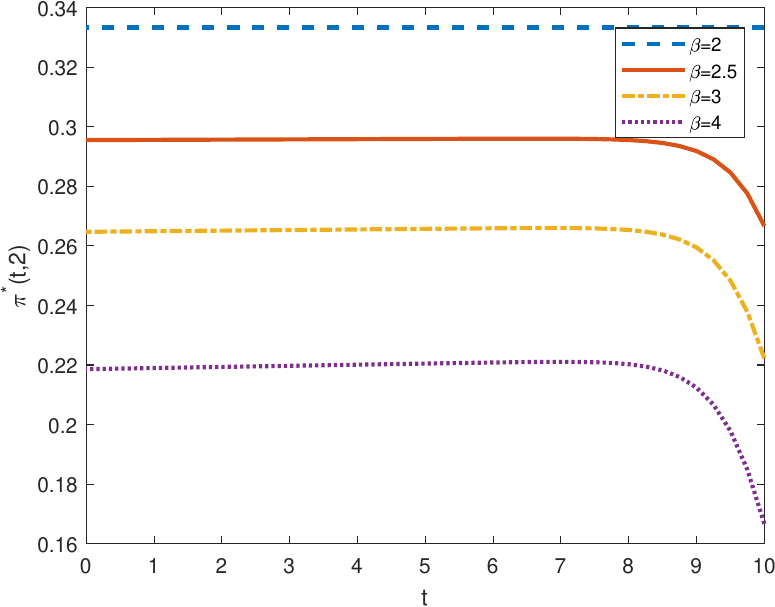}}
  \caption{Effects of $\beta$ ($\alpha=2$).}
  \label{Fig4}
\end{figure}
\begin{figure}[h]
  \centering
  \subcaptionbox{$\pi^*(t,1)$\label{fig11}}
    {\includegraphics[width=0.45\linewidth]{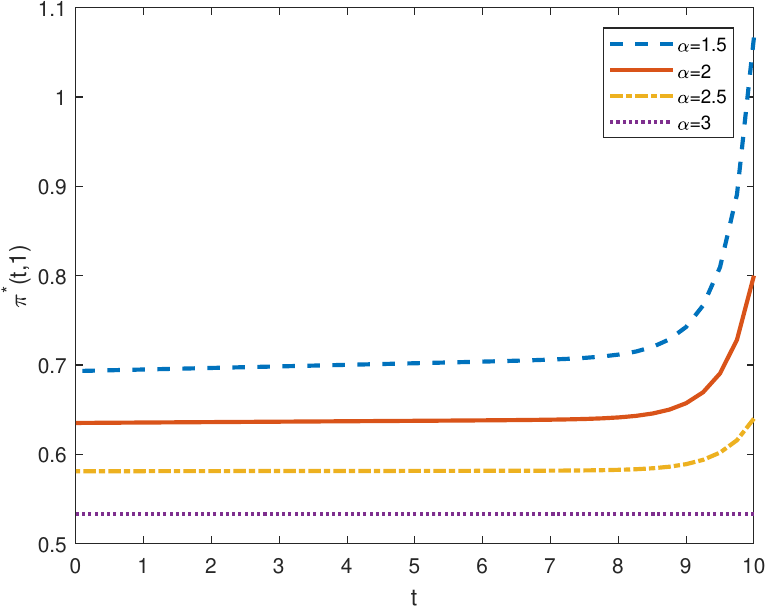}}
  \subcaptionbox{$\pi^*(t,2)$\label{fig12}}
    {\includegraphics[width=0.45\linewidth]{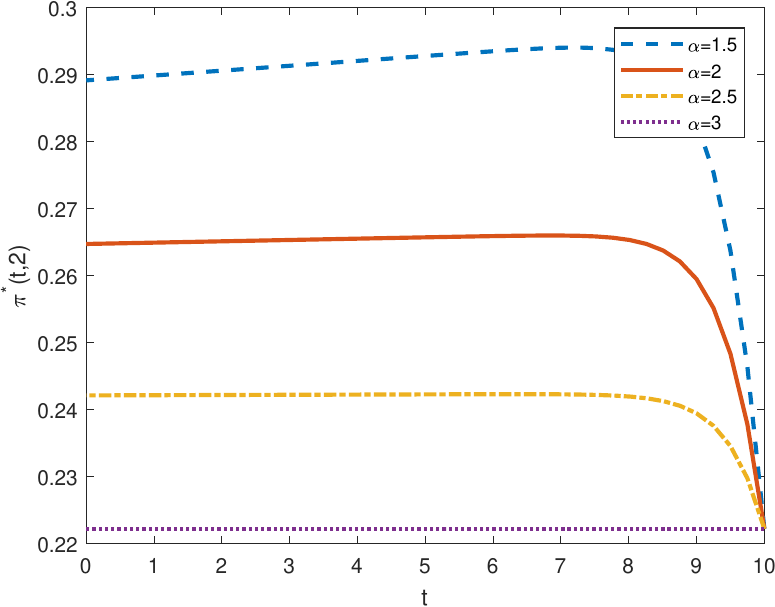}}
  \caption{Effects of $\alpha$ ($\beta=3$).}
  \label{Fig5}
\end{figure}
Figures \ref{Fig4} and \ref{Fig5} show that as the risk aversion rises, the equilibrium investment proportion in risky assets decreases, irrespective of whether the regime is a bull or bear market. This aligns with empirical observations.
\newpage
 \section{\bf Conclusion} 
This paper investigates an equilibrium portfolio selection problem in a continuous-time financial market with regime-switching, which is modeled by an observable continuous-time Markov process. Both market coefficients and the investor's utility functions depend on the prevailing market regime. We adopt an objective function that considers expected certainty equivalents, which provides better economic insight but leads to time-inconsistency. To address this issue, we introduce the concept of an equilibrium strategy, accompanied by Verification Theorem \ref{verification}. Unlike the verification theorem in \cite{desmettre2023equilibrium}, our work allows the regime to influence both market coefficients and the utility function, marking a significant contribution of this paper. Furthermore, we rigorously prove the verification theorem.

Specifically, we analyze the case of two market regimes using power utility functions (CRRA). By applying Verification Theorem \ref{verification}, employing separation of variables, and demonstrating the global existence of four-dimensional non-linear ODEs, we derive the feedback form of the equilibrium strategy. The global existence of the  ODEs represents a key contribution of this research, and we also  rigorously verify the equilibrium solution. Considering expected certainty equivalents, we maintain homogeneity in the beliefs-dependent utilities. Another mathematical contribution of this paper is the explicit solution to the optimization problem under beliefs-dependent utilities.

In our preference model of bull and bear markets, we show that the equilibrium investment strategy lies between two Merton's fractions. Consistent with the standard model of financial markets without regime-switching, the equilibrium investment proportion in the risky asset remains independent of current wealth. However, the equilibrium investment strategy is time-dependent.  The evolution of \(\pi^*\) over time \(t\) can be divided into two phases: when \(t\) is far from the terminal time \(T\), \(\pi^*\) remains relatively stable; as \(t\) approaches \(T\), \(\pi^*\) shifts rapidly to either \(\pi^*_1\) or \(\pi^*_2\), depending on the current regime.

\section*{Acknowledgments}
The authors acknowledge the support from the National Natural Science Foundation of China (Grant No.12271290, No.12371477). The authors also thank the members of the group of Actuarial Sciences and Mathematical Finance at the Department of Mathematical Sciences, Tsinghua University for their feedback and useful conversations. 

\bibliographystyle{plainnat}
\bibliography{references}
\end{document}